\documentclass[11pt]{article}
\usepackage[utf8]{inputenc}
\usepackage[margin=1in]{geometry}

\input xy
\xyoption{all}

\usepackage{amsmath,amsfonts,mathdots,mathrsfs,amsthm,amssymb,latexsym,amscd,amsopn,graphics,color,enumerate,lscape,mathtools}
\usepackage[linktocpage=true]{hyperref}
\usepackage{tikz-cd}
\usepackage{todonotes,comment}

\allowdisplaybreaks

\theoremstyle{plain}
  \newtheorem{thm}{Theorem}
  
  \newtheorem{conj}[thm]{Conjecture}
  \newtheorem{prop}[thm]{Proposition}
  \newtheorem{cor}[thm]{Corollary}
  \newtheorem{lem}[thm]{Lemma}

\theoremstyle{definition}
  \newtheorem{defn}[thm]{Definition}
  \newtheorem{exmp}[thm]{Example}
  
  \newtheorem{con}[thm]{Construction}
  \newtheorem{rem}[thm]{Remark}

\theoremstyle{remark}

\DeclareMathOperator{\Cl}{Cl}

\DeclareMathOperator{\Sym}{Sym}

\def\GL{{\rm GL}}
\def\SL{{\rm SL}}
\def\Cl{{\rm Cl}}

\def\Sym{{\rm Sym}}

\def\Vol{{\rm Vol}}

\newcommand{\SO}{\mathrm{SO}}
\newcommand{\Or}{\mathrm{O}}

\newcommand{\nc}{\newcommand}
\nc{\on}{\operatorname}
\nc{\renc}{\renewcommand}
\nc{\wt}{\widetilde}
\nc{\defeq}{\vcentcolon=}
\nc{\eqdef}{=\vcentcolon}
\nc{\Spec}{\on{Spec}}
\nc{\ol}{\overline}
\renc{\d}{\partial}

\newcommand{\Mod}[1]{\ (\mathrm{mod}\ #1)}

\title{Monogenic fields with odd class number Part II: even degree}

\author{Artane Siad}

\begin{document}

\maketitle

\begin{abstract}
In 1801, Gauss proved that there were infinitely many quadratic fields with odd class number. We generalise this result by showing that there are infinitely many $S_n$-fields of any given even degree and signature that have odd class number. Also, we prove that there are infinitely many fields of any even degree at least $4$ and with at least one real embedding that have units of every signature.

To do so, we bound the average number of $2$-torsion elements in the class group, narrow class group, and oriented class group of monogenised fields of even degree (and compute these averages precisely conditional on a tail estimate) using a parametrisation of Wood \cite{MelanieWoodRings, MelanieWoodIdealClasses}. These averages are the first $p$-torsion averages to be calculated for $p$ not coprime to the degree (in degree at least $3$), shedding light on the question of Cohen--Lenstra--Martinet--Malle type heuristics for class groups and narrow class groups at ``bad'' primes.
\end{abstract}

\tableofcontents

\section{Introduction}

In $1801$, Gauss introduced class groups in his \emph{Disquisitiones Arithmeticae} \cite{GaussDisquisitiones}, and posed what is recorded as the first question concerning their behaviour over families: are there infinitely many quadratic fields with class number $1$? This question, in the real quadratic case, is still open. Nevertheless, Gauss proved that there were infinitely many quadratic fields with odd class number. 

That result has been generalised, first to cubic fields by Bhargava \cite{BhargavaCohenLenstra}, and then to all odd degrees by Ho--Shankar--Varma \cite{HoShankarVarmaOdd}. 

\begin{thm}[Bhargava \cite{BhargavaCohenLenstra}, Ho--Shankar--Varma \cite{HoShankarVarmaOdd}]
For any odd degree $n$ and signature $(r_1,r_2)$, there are infinitely many fields of degree $n$ and signature $(r_1,r_2)$ that have odd class number. 
\end{thm}

As a corollary of our main theorem, we generalise Gauss's result to all fields of even degree. 

\begin{thm}\label{parity}
Let $n \ge 4$ be an even integer. For each choice of signature $(r_1,r_2)$, there are infinitely fields of degree $n$ and signature $(r_1,r_2)$ that have odd class number. 
\end{thm}

To obtain his result, Gauss calculated the size of the $2$-part of the narrow class group of quadratic fields. Knowing the $2$-part of the narrow class group gives information about unit signatures. As a corollary to our main theorem, we generalise to even degree the result of \cite{BhargavaVarma2Torsion} and \cite{HoShankarVarmaOdd} that in each odd degree, there are infinitely many fields which have units of every possible signature. 

\begin{thm}\label{units}
Let $n \ge 4$ be an even integer. For each choice of signature $(r_1,r_2)$, there are  infinitely fields of degree $n$ and signature $(r_1,r_2)$ that have units of units of every signature. 
\end{thm}

These results come from bounding the average size of the $2$-torsion part of the class group and narrow class group of monogenised fields of even degree. This means that we can add the adjectives monogenic and $S_n$ to Theorems \ref{parity} and \ref{units}. 

\subsection{Class group heuristics}

The Cohen--Lenstra--Martinet--Malle heuristics which were developed in a series of ground-breaking works \cite{CohenLenstraHeuristicsOnClassGroups,CohenMartinetEtudeHeuristique,CohenMartinetNotTooGood, CohenMartinetNumericalHeuristics,MalleDistributionClassGroups}, constitute our best conjectural description of the distribution of the $p^\infty$-part of the class group, $\Cl(K)[p^\infty]$, over families of number fields $K$ of fixed degree and signature ordered by discriminant for ``good'' primes $p$. We say that a prime $p$ is ``good'' if it is coprime to the degree of the field and ``bad'' otherwise.

Predictions arising from these heuristics have only been verified in two cases. Davenport and Heilbronn \cite{DavenportHeilbronn} calculated the average number of $3$-torsion elements in the class group of quadratic fields, and Bhargava \cite{BhargavaCohenLenstra} calculated the average number of $2$-torsion in the class group of cubic fields. The heuristics are expected to hold under any natural ordering on the family of fields, and not just when ordering by discriminant. In \cite{HoShankarVarmaOdd}, Ho--Shankar--Varma found evidence to support this expectation by showing that the average number of $2$-torsion elements in the class group of fields associated to binary $n$-ic forms, ordered either by naive height or by Julia invariant, coincided with the values predicted from the heuristics.

Remarkably, in all of the cases above (see also \cite{BhargavaVarma2Torsion}), the averages remain the same when one imposes finitely many local conditions or even an \emph{acceptable family} of local conditions. We call a set of local conditions \emph{acceptable} if for large enough primes $p$, it includes all fields with discriminant indivisible by $p^2$. It then becomes natural to ask about the effect of global conditions on the averages. Bhargava--Hanke--Shankar found in \cite{BhargavaHankeShankar} that monogenicity had the effect of doubling the average number of non-trivial elements in the $2$-torsion part of the class group of cubic fields! In Part I of this two-part series, we generalised this result to all odd degrees. 

\begin{thm}[\cite{SiadOddMonogenicAverages}]
Let $n \ge 3$ be an odd integers. Let $\mathfrak{R}$ be an acceptable family of monogenised fields ordered by naive height. The average number of $2$-torsion elements in the class group of fields in $\mathfrak{R}$ satisfies the bound:
$${\rm Avg}(\Cl_2,\mathfrak{R}) \le 1+ \frac{2}{2^{r_1+r_2-1}}$$
with equality conditional on a tail estimate. 
\end{thm}

Despite remarkable progress on refining, understanding and streamlining these heuristics (see \cite{MelanieWoodRandomIntegralMatrices,MelanieWeitong,MelanieWoodCohenLenstraAndLocalConditions,BartelArakelov,BartelClassGroupOfRandomNumberFields,TsimermanCohenLenstra1,TsimermanCohenLenstra2}), nothing has been proposed so far to describe the distribution of $p$-torsion in the class group for ``bad'' primes $p$. In particular, there is no prediction for the average number of the $2$-torsion elements in the class group, narrow class group or oriented class group of number fields of even degree. 

The most serious reason for this gap seems to be the presence of \emph{genus theory}. Indeed, it is unclear how the $1/{\rm Aut}(\cdot)$ count of Cohen--Lenstra mixes with the ``deterministic'' input from genus theory. Even in the quadratic case, Gerth conjectured \cite{Gerth} and Smith proved \cite{AlexSmith}, that one needs to throw out the $2$-part to recover a $1/{\rm Aut}(\cdot)$ count. A less severe obstacle concerns the presence of $2$-nd roots of unity in the base field. Malle showed that roots of unity in the base field affected Cohen--Lenstra averages and it is unknown to what extent this phenomenon intervenes when genus theory is involved. \\

We now clarify what we mean by \emph{genus theory}. In \emph{Disquisitiones Arithmeticae} \cite{GaussDisquisitiones}, Gauss determined the structure of the $2$-torsion part of the narrow class group of quadratic fields by relating it to binary quadratic forms with a composition law. 

\begin{thm}[Gauss's genus theory \cite{GaussDisquisitiones}]
Let $K$ be a quadratic field, $\Delta_{K / \mathbb{Q}}$ its discriminant and $\Cl^+_2[K]$ the $2$-torsion in its narrow class group. Let $\omega(\Delta_{K / \mathbb{Q}})$ denote the number of distinct prime factors of $\Delta_{K / \mathbb{Q}}$. Then the $2$-torsion in the narrow class group of $K$ is given by $$\Cl^+_2[K] \cong \left( \frac{\mathbb{Z}}{2\mathbb{Z}} \right)^{\omega(\Delta_{K / \mathbb{Q}})-1}  .$$
\end{thm}

The basic idea behind this theorem is that a ramified prime $q$ contributes a $2$-torsion element to the narrow class group since it must split as $(q) = \mathfrak{a}^2$ for some prime ideal $\mathfrak{a}$. It is this feature that we call \emph{genus theory} in the context of Cohen--Lenstra averages. More precisely, if $p$ is a ``bad'' prime, then primes $q$ which ramify as $(q) = \mathfrak{a}^p$ for some ideal $\mathfrak{a}$ are a source of $p$-torsion elements in the class group. \\

In this paper, we bound the average number of $2$-torsion elements in the \emph{class group}, \emph{narrow class group} and \emph{oriented class group} of monogenised fields of even degree at least $4$ (unramified at $2$) (and compute it precisely conditional on a tail estimate). These averages are the first of their kind (in degree at least $3$) to be calculated in a setting where genus theory, roots of unity, and global conditions all play a part.

\subsection{Preliminary definitions} A number field $K$ of degree $n$ is said to be \emph{monogenic} if $\mathcal{O}_K = \mathbb{Z}[\alpha]$ for some $\alpha \in \mathcal{O}_K$. The element $\alpha$ is called a \emph{monogeniser} of the field $K$. A \emph{monogenised} field is the data $(K,\alpha)$ of a monogenic field together with a choice of monogeniser. It is known that a monogenic field has finitely many monogenisers up to transformations of the form $\alpha \mapsto \pm \alpha + m$ for some $m \in \mathbb{Z}$. Futhermore, it is expected that $100\%$ of monogenic fields possess a unique monogeniser up to transformations of the form $\alpha \mapsto \pm \alpha + m$ for some $m \in \mathbb{Z}$, see \cite{BhargavaShankarWangSquarefreeI}. This motivates the following definition. 

\begin{defn}
Two monogenised fields $(K,\alpha)$ and $(K',\alpha')$ are said to be isomorphic if there exists a field isomorphism from $K$ to $K'$ taking $\alpha$ to $\pm \alpha' +m$ for some $m \in \mathbb{Z}$.
\end{defn}

Thus, we expect monogenised fields to be statistically equivalent to monogenic fields when computing averages. Nevertheless, suppose a statement holds for a ``positive proportion'' of \emph{monogenised} fields. In that case, the same statement is true for ``infinitely many'' \emph{monogenic} fields by using the arguments of \cite{HoShankarVarmaOdd} combined with the construction of strongly quasi-reduced elements of \cite{BhargavaShankarWangSquarefreeI}.

The height we choose for monogenised fields has a convenient interpretation. Each isomorphism class of monogenised field contains a unique element $(K,\alpha_0)$ with the property that $0 \le {\rm tr}(\alpha_0) < n$. If $f(x) = x^n+a_1 x^{n-1}+\ldots+a_n$ is the minimal polynomial of $\alpha_0$, we define the \emph{naive height} of the isomorphism class to be: $$H\Big( \left[(K,\alpha_0)\right]\Big) = \max_i \Big\{|a_i|^{1/i} \Big\}.$$

We denote by $\mathfrak{R}^{r_1,r_2}$ the collection of isomorphism classes of monogenised $S_n$-fields of signature $(r_1,r_2)$ ordered by naive height. In Section \ref{The parametrisations}, we will see that the set of monogenised fields is in natural bijection with the set of monic degree $n$ polynomials. This bijection equips the set of monogenised fields with a natural local measure, and we can speak of families of fields in $\mathfrak{R}^{r_1,r_2}$ associated with sets of local specifications $(\Sigma_p)_p$ on monic polynomials. 

We will also need the notion of oriented class groups to state our results. An \emph{oriented ideal} is a pair  $(I,\varepsilon)$ consisting of a fractional ideal $I$ together with an orientation $\varepsilon = \pm 1$. We say that an oriented ideal is a \emph{principal oriented ideal} if it is of the form $((\alpha), {\rm sgn}(N(\alpha)))$. 

\begin{defn}
The \emph{oriented class group}, $\Cl^*(K)$, of a number field $K$ consists of the set of oriented fractional ideals of $K$ modulo the principal oriented fractional ideals. The operation is component-wise multiplication.
\end{defn}

Recall that $\Cl^*(K)$ is isomorphic to the usual class group of $K$ if $K$ has a unit of negative norm and is a $\mathbb{Z}/2\mathbb{Z}$ extension of the usual class group if $K$ does not have a unit of negative norm, \cite{BhargavaGrossWang}. Notice that when $K$ does not have a unit of negative norm, $2$-torsion in the oriented class group is not always twice as big as $2$-torsion in the usual class group. The reason is that some $4$-torsion elements in the oriented class group could map to $2$-torsion elements in the ordinary class group under the forgetful map. 

\subsection{Outline of the results} 

To quantify the contribution of genus theory in each of our averages, we introduce a local quantity that tracks the proportion of fields in the family which are \emph{evenly ramified} at $p$, i.e. where $(p) = \mathfrak{a}^2$ for some ideal $\mathfrak{a}$.

\begin{defn}
Let $\mathfrak{R} \subset \mathfrak{R}^{r_1,r_2}$ be a family of fields corresponding to an acceptable family of local specifications $\Sigma = (\Sigma_p)_p$. We define $r_p(\mathfrak{R})$, the \emph{even ramification density} of $\mathfrak{R}$ at $p$, as the density in $\Sigma_p$ of elements of $\Sigma_p$ which are evenly ramified at $p$.
\end{defn}

\begin{thm}[Main theorem]
Let $\mathfrak{R} \subset \mathfrak{R}^{r_1,r_2}$ be a family of fields (unramified at $2$ and with local conditions at $2$ given modulo $2$) corresponding to an acceptable family of local specifications $\Sigma = (\Sigma_p)_p$ and let $r_p(\mathfrak{R})$ denotes its even ramification density at $p$. 

If $r_1 = 0$, the average number of $2$-torsion elements in the class group, narrow class group, and oriented class group of fields in $\mathfrak{R}$ satisfies the bound:
\begin{equation}
{\rm Avg}(\Cl_2,\mathfrak{R}) = {\rm Avg}(\Cl_2^+,\mathfrak{R}) = \frac{1}{2} {\rm Avg}(\Cl_2^*,\mathfrak{R}) \le \prod_{p \neq 2} (1+r_p(\mathfrak{R}))\left(1+\frac{2}{2^{r_2}}\right) + \frac{1}{2^{r_2}} \quad \\
\end{equation}
 with equality conditional on tail estimate. 

If $r_1 > 0$, the average number of $2$-torsion elements in the oriented class group of fields in $\mathfrak{R}$ satisfies the bound: 
\begin{equation}
{\rm Avg}(\Cl_2^*,\mathfrak{R}) \le \prod_{p \neq 2} (1+r_p(\mathfrak{R}))\left(1+\frac{2}{2^{r_1+r_2-1}}\right) + \frac{1}{2^{r_1+r_2-1}} \quad \\
\end{equation}
 with equality conditional on tail estimate. 

If $r_1 > 0$, the average number of $2$-torsion elements in the narrow class group of fields in $\mathfrak{R}$ satisfies the bound: 
\begin{equation}
{\rm Avg}(\Cl_2^+,\mathfrak{R}) \le \prod_{p \neq 2} (1+r_p(\mathfrak{R}))\left(1+\frac{2}{2^{\frac{n}{2}}}\right) + \frac{1}{2^{r_2}} \quad  \\
\end{equation}
with equality conditional on tail estimate. 

If $r_1 > 0$, the average number of $2$-torsion elements in the class group of fields in $\mathfrak{R}$ satisfies the bound: 
\begin{equation}
{\begin{split}
{\rm Avg}(\Cl_2,\mathfrak{R}) &\le \frac{1}{2}\prod_{p \equiv 1 \textrm{ mod } 4} (1+r_p(\mathfrak{R})) \left( \prod_{p \equiv 3 \textrm{ mod } 4}  (1-r_p(\mathfrak{R})) + \prod_{p \equiv 3 \textrm{ mod } 4}  (1+r_p(\mathfrak{R})) \right) \\ 
&+ \frac{1+2\prod_{p \neq 2} (1+r_p(\mathfrak{R}))}{2^{r_1+r_2}}
\end{split} }\quad  \\
\end{equation}
 with equality conditional on tail estimate. 
\end{thm}

These averages have several interesting consequences, two of which were mentioned earlier. First, they generalise a theorem of Gauss stating that there are infinitely many quadratic fields with odd class number. For odd degree, Ho--Shankar--Varma proved in \cite{HoShankarVarmaOdd} that there are infinitely many fields of fixed odd degree and signature $(r_1,r_2)$ with odd class number. For even degree, the corresponding problem has remained open for degrees strictly larger than $4$. In degree $4$, the result is known for biquadratic fields (see Koymans--Pagano in \cite{Highergenustheory}) but not for $S_4$-fields to the author's knowledge. 

\begin{cor}
Let $n \ge 4$ be an even integer. For each choice of signature $(r_1,r_2)$, there are infinitely many degree $n$ monogenic $S_n$-fields with signature $(r_1,r_2)$ that have odd class number. 
\end{cor}

Second, the averages show that there are infinitely many even degree fields with units of every signature. 

\begin{cor}
Let $n \ge 4$ be an even integer. For each choice of signature $(r_1,r_2)$, there are infinitely many degree $n$ monogenic $S_n$-fields with signature $(r_1,r_2)$ that have units of units of every signature. 
\end{cor}

Third, conditional on a tail estimate, they show that genus theory is the only added complexity to consider for ``bad'' primes. Indeed, the formula for ${\rm Avg}(\Cl_2,\mathfrak{R})$ when all $r_p(\mathfrak{R}) = 0$ is only slightly different from the one for the average $2$-torsion in the class group of monogenised fields of odd degree \cite{SiadOddMonogenicAverages}.

\begin{cor}
Let $n \ge 4$ be an even integers. Let $\mathfrak{R} \subset \mathfrak{R}^{r_1,r_2}$ be a family of fields (unramified at $2$ and with local conditions at $2$ given modulo $2$) corresponding to an acceptable family of local specifications $\Sigma = (\Sigma_p)_p$ and such that the even ramification density at all primes is zero, $r_p(\mathfrak{R})=0$. 

The average number of $2$-torsion elements in the class group of fields in $\mathfrak{R}$ satisfies the bound: $${\rm Avg}(\Cl_2,\mathfrak{R}) \le 1+ \frac{3}{2^{r_1+r_2}},$$
with equality conditional on a tail estimate. 
\end{cor}

Lastly, they offer hints as to the ``right answer'' for Cohen--Lenstra averages for ``bad'' primes.  

\begin{rem} 
The assumption that the families considered are unramified at $2$ and have local conditions at $2$ given modulo $2$ is an artefact of the particulars of the mass computation at $2$ and does not change the validity of the ``positive proportion'' versions of the corollaries above. We expect to be able to remove it in the future and to be able to upgrade the theorems to compute the average value of $\left| \Cl_2(\mathcal{O}) \right| - \frac{1}{2^{r_1+r_2}}\left|\mathcal{I}_2(\mathcal{O}) \right|$  for orders which are not necessarily maximal. 
\end{rem}
 
\subsection{Strategy} 

The strategy used to prove the main theorem is roughly the same as that used in \cite{SiadOddMonogenicAverages}. We apply Wood's parametrisation \cite{MelanieWoodRings, MelanieWoodIdealClasses} of $2$-torsion ideal classes in rings associated to monic binary forms in terms of integral orbits for certain representations to reduce the question to an asymptotic counting problem in the geometry of numbers. The calculation reduces to counting integral orbits for a group acting on a variety. However, unlike in the odd degree case \cite{SiadOddMonogenicAverages}, a slight change in the choice of group and variety leads to a significant difference in the arithmetic information that is encoded. This feature is unique to the even degree case and arises because there are fields of even degree which don't possess units of negative norm (in odd degree, the unit $-1$ has norm $-1$). 

We begin by describing the case of the \emph{oriented class group}. When computing averages for $2$-torsion in the oriented class group, the relevant space is the set of $\SL_n(\mathbb{Z})$-orbits of pairs of integral symmetric matrices $(A,B)$ with the constraint $\det(A)=(-1)^{\frac{n}{2}}$. To apply the geometry of numbers, we borrow the idea of \cite{BhargavaHankeShankar} and ``linearise'' the problem by noting that up to $\SL_n(\mathbb{Z})$, there are only finitely many equivalence classes of symmetric integral matrices of determinant $(-1)^{\frac{n}{2}}$. We denote this finite collection by $\mathscr{L}_\mathbb{Z}$. Counting $\SL_n(\mathbb{Z})$ orbits on the space of pairs $(A,B)$ with the constraint $\det(A)=(-1)^{\frac{n}{2}}$ is thus reduced to counting $\SO_{A_0}(\mathbb{Z})$ orbits on the space of pairs $(A_0,B)$. We handle this count in the usual way, and we get the inequality of the main theorem conditional on a tail estimate. A new feature unique to the even degree case appears when computing the local masses. There is extra local mass at primes which are evenly ramified!

Now, to compute the average $2$-torsion in the \emph{class group}, one cannot directly use the results for the oriented class group. Indeed, when $K$ does not have a unit of negative norm, some $4$-torsion elements in the oriented class group map to $2$-torsion elements in the usual class group under the forgetful map. These elements are captured by pairs $(A,B)$ where $\det(A)=-1 \cdot (-1)^\frac{n}{2}$. This means that we need to run the strategy above on $\SL_n^\pm$-orbits on the space of pairs $(A,B)$ of bilinear forms with $\det(A)=\pm 1$. A new feature appears in the final calculation because the local masses for $\det(A) =  (-1)^\frac{n}{2}$ and $\det(A) = -1 \cdot (-1)^\frac{n}{2}$ behave differently at primes congruent to $3 \textrm{ mod } 4$. This is why the averages for the class group look more complicated than those for the oriented and narrow class groups.

\subsection{Organisation of the paper} In Section \ref{The parametrisations} we recall the parametrisation of $2$-torsion ideal classes in the oriented class group in terms of $\SL_n$-orbits on pairs integral symmetric matrices. This parametrisation already appears in \cite{BhargavaGrossWang}, albeit with a small mistake which we correct. In Sections \ref{Reduction theory}-\ref{Change of measure formula}, we run the geometry of numbers arguments to obtain an asymptotic formula for the number of $2$-torsion elements in the oriented class group of monogenised fields of height at most $X$ in terms of a product of local volumes dependent on $X$. Because we use a square-free sieve dependent on a conjectural tail estimate to bound the $2$-torsion from below in terms of the product of local volumes, this equality is conditional. In Section \ref{The product of local volumes and the total local mass} we compute the total local mass. It is there that we find genus theory at play in the guise of extra orbits at evenly ramified primes. In Sections \ref{Point count and the $2$-adic mass} and Section \ref{The infinite mass} we see how the total local masses distribute among the $(A_0,B)$ slices by applying the equidistribution techniques of \cite{SiadOddMonogenicAverages}. Lastly, we complete the proof of the main theorem for the oriented class group in Section \ref{Statistical consequences} by summing the $A_0$ counts over all $A_0 \in \mathscr{L}_{\mathbb{Z}}$. In Section \ref{Averages for the class group and narrow class group}, we repeat the steps above to prove the main theorem for class groups and narrow class groups. We can recycle most of the counting arguments, but the parametrisation and the final count are much more subtle.

\subsection{Acknowledgements} It is pleasure to thank Arul Shankar for suggesting the problem and for many helpful conversations. I am grateful to Manjul Bhargava, Jonathan Hanke, Ashvin Swaminathan, and Ila Varma for useful discussions and many helpful comments. Many thanks to Be\~{n}at Mencia Uranga and Simon Lieu of the Cavendish Laboratory for their hospitality during the summer of $2019$. The author's research was supported by a McCuaig--Throop Bursary and an Ontario Graduate Scholarship.

\section{The parametrisations} \label{The parametrisations}

\subsection{The parametrisation of monogenised $n$-ic rings}\label{parametrisationrings} 
In order to count the number of monogenised $n$-ic rings having bounded height, we will use the following parametrisation in terms of binary $n$-ic forms: 

\begin{defn}
Let $U = {\rm Sym}_n(2)$ denote the space of binary $n$-ic forms. We denote by $U_1 \subset U$ the space of all monic binary $n$-ic forms $f(x,y) = x^n+a_{n-1}x^{n-1}y+\ldots+a_0y^n.$ The group ${\rm GL}_2$ acts on $U$ via the twisted action $\gamma \cdot f(x,y) := \det(\gamma)^{-1}f((x,y)\cdot \gamma)$ for $\gamma \in {\rm GL}_2$ and $f \in U$. Let $F \subset {\rm GL}_2$ denote the group of lower triangular unipotent matrices. Then the action of $F$ on $U$ preserves $U_1$ and yields an action of $F$ on $U_1$.
\end{defn}

We say that a pair $(R,\alpha)$ is a monogenised $n$-ic ring if $R$ is an $n$-ic ring and $\alpha$ is an element of $R$ such that $R = \mathbb{Z}[\alpha]$. Two monogenised $n$-ic rings $(R,\alpha)$ and $(R,\alpha')$ are said to be isomorphic if $R$ and $R'$ are isomorphic via a ring isomorphism sending $\alpha$ to $\alpha'+m$ for some $m \in \mathbb{Z}$. We then have the following explicit parametrisation of monogenised $n$-ic rings in terms of the orbit data introduced above: 

\begin{thm}
There is a natural bijection between isomorphism classes of monogenised $n$-ic rings and $F(\mathbb{Z})$-orbits on $U_1(\mathbb{Z})$. 
\end{thm}
\begin{proof} Consider the map sending a monic binary $n$-ic form $f(x,y) \in U_1(\mathbb{Z})$ to the monogenised $n$-ic ring $R_f := \left(\frac{\mathbb{Z}[\theta]}{(f(\theta,1))},\theta \right)$. This map descends to a map from $F(\mathbb{Z}) \backslash U_1(\mathbb{Z})$ to isomorphism classes of monogenised $n$-ic rings which we denote by $\Phi$. Indeed, if $g = \gamma \cdot f$ for $\gamma = \bigl[ \begin{smallmatrix} 1&0\\ m&1 \end{smallmatrix} \bigr] \in F(\mathbb{Z})$, then $g(\theta,1) = f(\theta+m,1)$ and the monogenised ring $\left(\frac{\mathbb{Z}[\theta]}{(f(\theta,1))},\theta \right)$ is isomorphic to the monogenised ring $\left(\frac{\mathbb{Z}[\theta]}{(f(\theta+m,1))},\theta\right)$ through $\theta \mapsto \theta+m$. To verify that $\Phi$ is surjective, note that it was already surjective as a map from monic binary $n$-ic forms to monogenised $n$-ic rings. To verify that $\Phi$ is injective, suppose that $\Phi(f) = \left(\frac{\mathbb{Z}[\theta]}{(f(\theta,1))},\theta\right)$ is isomorphic to $\Phi(g) = \left(\frac{\mathbb{Z}[\omega]}{(g(\omega,1))},\omega\right)$. Then $\theta \mapsto \omega + m$ for some $m \in \mathbb{Z}$ under this isomorphism. Consequently, $f(\theta,1) = 0$ in $\Phi(f)$ means that $f(\omega+m,1) = 0$ in $\Phi(g)$. In other words, the polynomial $g(\omega,1)$ divides $f(\omega+m,1)$. But since both are monic, we must have $g(\omega,1) = f(\omega+m,1)$. Thus, $g = \bigl[ \begin{smallmatrix} 1&0\\ m&1 \end{smallmatrix} \bigr] \cdot f $ and $g = f$ in $F(\mathbb{Z}) \backslash U_1(\mathbb{Z})$.
\end{proof}

\subsection{Orbits of pairs of symmetric bilinear forms}

We define the space of pairs of symmetric matrices as well as the resolvent map.  

\begin{defn}
Let $T$ be a base ring. Let $$V(T) = T^2 \otimes \Sym^2 (T^n)$$ be the space of pairs of symmetric $n \times n$ matrices with coefficients in $T$. The group $\GL_n(T)$ acts on $V(T)$ by change of basis. In other words, if $\gamma \in \GL_n(T)$ and $(A,B) \in V(T)$, we define $$\gamma (A,B) = (\gamma^t A \gamma, \gamma^t B \gamma),$$ where $\gamma^t$ denotes the transpose of $\gamma$. 
\end{defn}

There is a natural map from this space of pairs of matrices, $V(T)$, to the space of polynomials, $U(T)$, called the resolvent map.

\begin{defn}[The resolvent map $\pi$]
Let $T$ be a base ring. We define the resolvent map $\pi \colon V(T) \rightarrow U(T)$ by $$(A,B) \mapsto {\rm disc}(Ax-By) = (-1)^{\frac{n}{2}}\det(Ax-B).$$
We say that a pair $(A,B) \in V(T)$ is non-degenerate if the associated binary form $f_{(A,B)}$ is non-degenerate (has non-zero discriminant). 
\end{defn}

\begin{exmp} 
In contrast to the odd case, when the degree $n$ is even, there might not
 exist for every $f \in \mathbb{Z}[x,y]$ a pair $(A,B) \in V(\mathbb{Z})$ whose resolvent polynomial is $f$. For instance, the binary form $-x^2-y^2$ is not resolvent polynomial of any element of $V(\mathbb{R})$. For monic $f$, however, the torsor is always non-empty! For example, $x^2+y^2$ is the resolvent of $\left( \left[ \begin{smallmatrix} 0 & 1 \\ 1 & 0 \end{smallmatrix} \right], \left[ \begin{smallmatrix} 1 & 0 \\ 0 & -1 \end{smallmatrix} \right] \right)$ and $x^2-y^2$ is the resolvent of  $\left( \left[ \begin{smallmatrix} 0 & 1 \\ 1 & 0 \end{smallmatrix} \right], \left[ \begin{smallmatrix} 1 & 0 \\ 0 & 1 \end{smallmatrix} \right] \right)$.
\end{exmp} 

We now describe the rigid parametrisation of pairs of symmetric bilinear forms.

\subsection{Rigid parametrisation of pairs of symmetric bilinear forms}
Let $T$ be a principal ideal domain. Recall the rigid parametrisation of the pairs of bilinear forms $(A,B) \in V(T)$ with resolvent polynomial equal to $f$ in terms of the based fractional ideal data for $R_f$.

\begin{thm}[\cite{MelanieWoodIdealClasses}]
Take a non-degenerate binary $n$-ic form $f \in U(T)$ and let $R_f = \frac{T[x]}{(f(x))}$. Then the pairs symmetric bilinear forms $(A,B) \in V(T)$ with $f_{(A,B)} = f$ are in bijection with equivalence classes of triples $$(I, \mathcal{B},\delta)$$ where $I \subset K_f$ is a based fractional ideal of $R_f$ with basis $\mathcal{B}$ given by a $T$ module isomorphism $\mathcal{B} \colon I \rightarrow T^n$, $\delta \in K_f^\times$ such that $I^2 \subset \delta R_f^{n-3}$ as ideals and the norm equation holds $N(I)^2 = N(\delta) N(R_f^{n-3})$ (as based ideals). Two triples $(I,\mathcal{B},\delta)$ and $(I',\mathcal{B}',\delta')$ are equivalent if there exists a $\kappa \in K_f^\times$ with the property that $I = \kappa I'$, $\mathcal{B} \circ (\times \kappa) = \mathcal{B}'$ and $\delta = \kappa^2 \delta$.
\end{thm}

\subsection{$\SL_n(\mathbb{Z})$-orbits and the oriented class group}

We now let $\SL_n(T)$ act on $V$ by change of basis. This action corresponds, at the level of equivalence classes of triples $(I,\mathcal{B},\delta)$, to the action of $\SL_n(T)$ on the basis $\mathcal{B}$ of the fractional ideal $I$. This action only changes the basis and does not change $I$, $\delta$, nor the orientation of $\mathcal{B}$. In particular, it does not change the defining conditions: 
\begin{enumerate}
\item $I^2 \subset \delta R_f^{n-3}$
\item $N(I)^2 = N(\delta)N(R_f^{n-3})$
\end{enumerate} 

The rigid parametrisation can be used to describe $\SL_n(\mathbb{Z})$ orbits.  

\begin{thm}[\cite{BhargavaGrossWang}]
Let $f$ be a non-degenerate monic binary $n$-ic form $f \in U_1(\mathbb{Z})$, let $R_f = \frac{\mathbb{Z}[x]}{(f(x))}$ and $K_f = R_f \otimes_{\mathbb{Z}} \mathbb{Q}$. The $\SL_n(\mathbb{Z})$ orbits on pairs symmetric bilinear forms $(A,B) \in V(\mathbb{Z})$ with $f_{(A,B)} = f$ are in bijection with equivalence classes of triples $$(I,\varepsilon,\delta)$$ where $I \subset K_f$ is a fractional ideal of $R_f$, $\varepsilon = \pm 1$ indicates the orientation of $I$, and $\delta \in K_f^\times$ is such that $I^2 \subset \delta R_f$ as ideals and the norm equation $N(I)^2 = N(\delta)$ holds (as oriented ideals). Two triples $(I,\varepsilon,\delta)$ and $(I',\varepsilon',\delta')$ are equivalent if there exists a $\kappa \in K_f^\times$ with the property that $I = \kappa I'$, $\varepsilon = {\rm sgn}(N(\kappa)) \varepsilon'$ and $\delta = \kappa^2 \delta'$.
\end{thm}

We recall the definition of the \emph{oriented class group} of an order in a number field.

\begin{defn}
The \emph{oriented class group}, $\Cl^*(\mathcal{O})$, of an order $\mathcal{O}$ consists of the set of oriented fractional ideals of $\mathcal{O}$ modulo the principal oriented fractional ideals of $\mathcal{O}$. We recall that $\Cl^*(\mathcal{O})$ is isomorphic to the usual class group of $\mathcal{O}$ when $\mathcal{O}$ has a unit of negative norm and is $\mathbb{Z}/2\mathbb{Z}$ extension of the usual class group of $\mathcal{O}$ otherwise. 
\end{defn}

We now describe the relation between $\SL_n(\mathbb{Z})$ orbits and the $2$-torsion part of the oriented class group of $R_f$. 

\begin{defn} 
A triple pair $(I,\varepsilon,\delta)$ is said to be projective if the ideal $I$ is invertible. Equivalently, a triple $(I,\varepsilon,\delta)$ is projective if $I^2 = (\delta)$. For an $S_n$ order $\mathcal{O}$, we write $H^*(\mathcal{O})$ for the set of oriented projective ideals of $\mathcal{O}$. 
\end{defn}

As the oriented class group is comprised of invertible oriented ideals, let us consider the restriction of the parametrisation above to the set $H^*(\mathcal{O})$. Consider the forgetful map:
\begin{equation*}
H^*(\mathcal{O}) \longrightarrow \Cl_2^*(\mathcal{O}).
\end{equation*}
Elements $(I_0,\varepsilon_0,\delta_0)$ in the kernel of this forgetful map have the property that $I_0 = (\alpha)$ for some $\alpha$ in $K$ as oriented ideals. Thus, $(I_0,\varepsilon_0,\delta_0) \sim ((\alpha),\varepsilon_0,\delta_0) \sim (\mathcal{O},1,\alpha^{-2}\delta_0)$. Now, $\mathcal{O} \subset \alpha^{-2}\delta_0 \mathcal{O}$ and $1 = N(\alpha^{-2} \delta_0)$. This implies that $(I_0,\varepsilon_0,\delta_0)$ is equivalent to $(\mathcal{O},1,u)$ where $u$ is norm $1$ unit of $\mathcal{O}^{\times}$. We are allowed to mod out $u$ by squares of units of $K^\times$ which fix the oriented ideal $(\mathcal{O},1)$. But those are precisely the norm $1$ units of $\mathcal{O}$. The sequence above can thus be completed to a short exact sequence: 
\begin{equation*}
1 \longrightarrow \frac{\mathcal{O}^{\times}_{N \equiv 1}}{(\mathcal{O}_{N \equiv 1}^{\times})^2} \longrightarrow H^*(\mathcal{O}) \longrightarrow \Cl_2^*(\mathcal{O}) \longrightarrow 1.
\end{equation*}

Applying Dirichlet's unit theorem allows us to compute that $100\%$ of the time (since having a non-trivial torsion is rare):
\begin{equation*} 
\left|\frac{\mathcal{O}^{\times}_{N \equiv 1}}{(\mathcal{O}_{N \equiv 1}^{\times})^2} \right| = 
\begin{cases}
2^{r_1+r_2} & \text{ if } \mathcal{O} \text{ has a unit of norm } -1\\
2^{r_1+r_2} & \text{ otherwise }
\end{cases}.
\end{equation*} 

We thus obtain the following formula for the number of elements in $H^*(\mathcal{O})$.

\begin{lem}
Let $\mathcal{O}$ be an order in an $S_n$-number field of degree $n$ and signature $(r_1,r_2)$. Then: 
\begin{equation*}
\left|H^*(\mathcal{O})\right| = 2^{r_1+r_2} \left|\Cl_2^*(\mathcal{O}) \right|. 
\end{equation*}
\end{lem}

We now proceed to the description of the points in the cuspidal regions. 

Here is the version of the reducibility criterion in Ho--Shankar--Varma which is valid in even degree. Roughly, for $(A,B)$ to be reducible, all we need is to have the two components of some $\SL_n(\mathbb{Q})$ translate of $(A,B)$ contain a common subsquare which is zero except at possibly one entry on the diagonal. 

\begin{lem}[Even degree distinguished orbits lemma] 
$(A,B) \in V(\mathbb{Q})$ is distinguished if and only if there is an $\SL_n(\mathbb{Q})$ translate of $(A,B)$ with the property that $$a_{i,j} = b_{i,j} = 0$$ for all $1 \le i,j \le \frac{n}{2}$ except for $i = j = \frac{n}{2}$ for which $a_{\frac{n}{2}\, \frac{n}{2}} = 0 \neq b_{\frac{n}{2}\, \frac{n}{2}}$.
\end{lem}
\begin{proof}
By \cite{ShankarWang2Selmer}, we know that a self adjoint operator in $V_f(\mathbb{Q})$ is distinguished if and only if there exists a $\mathbb{Q}$ rational $\frac{n-2}{2}$ plane $X$ such that ${\rm Span} \{X, TX\}$ is an isotropic $\frac{n-2}{2}+1$ plane. It thus suffices to translate this condition into the second condition in the statement of the lemma. 

First, let's suppose that $(A,B)$ are such that $a_{i,j} = b_{i,j} = 0$ for all $1 \le i,j \le \frac{n}{2}-1$ except for $i=j=\frac{n}{2}$ for which $a_{\frac{n}{2} \,\frac{n}{2}} = 0 \neq b_{\frac{n}{2} \,\frac{n}{2}}$. Let $\{e_1,\ldots, e_\frac{n}{2}, f_1,\ldots, f_{\frac{n}{2}} \}$ be the standard basis and let $T = A^{-1}B$. Then, let us set $X = \{e_1,\ldots, e_{\frac{n}{2}-1}\}$. We claim that ${\rm Span} \{X, TX\}$ is an isotropic $\frac{n}{2}$ plane. Indeed, $B(e_i) \in {\rm Span}\{f_1,\ldots,f_{\frac{n}{2}}\}$ for $1 \le i \le \frac{n}{2}-1$ and thus $A^{-1}B(e_i) \in {\rm Span}\{e_1,\ldots,e_{\frac{n}{2}}\}$. Therefore, $(A,B)$ is distinguished in that case. 

The other direction is easy. For instance, it follows directly from the arguments of \cite{ShankarWang2Selmer} or a direct calculation from Melanie Wood's parametrisation. 
\end{proof}

\begin{defn}
Let $\mathcal{O}$ be an order. We denote by $\mathcal{I}_2^*(\mathcal{O})$ the $2$-torsion subgroup of the oriented ideal group of $\mathcal{O}$. 
\end{defn} 

\begin{prop}
Let $\mathcal{O}_f$ be an order corresponding to the integral primitive irreducible non-degenerate monic binary form $f$. Then, $\mathcal{I}_2^*(\mathcal{O}_f)$ is in natural bijection with the set of projective reducible $\SL_n(\mathbb{Z})$-orbits on $V(\mathbb{Z}) \cap \pi^{-1}(f)$. 
\end{prop}

In particular, for maximal rings, there are always exactly $2$ projective reducible orbits. This will come into play when we calculate the final averages in Section \ref{Statistical consequences}.

\subsection*{$\SL_n$-orbits over fields and local rings} In this section, we compute the number orbits, and the size of the stabilisers for the action of $\SL_n$ on the arithmetic rings $\mathbb{Z}_p$, $\mathbb{Q}$, and $ \mathbb{R}$. Let $T$ be a principal ideal domain. We first restate the rigid parametrisation in this case.
\begin{thm}[\cite{BhargavaGrossWang}]
Let $f$ be a non-degenerate monic binary $n$-ic form $f \in U_1(T)$, let $R_f = \frac{T[x]}{(f(x))}$ and $K_f = R_f \otimes_{\mathbb{Z}} \mathbb{Q}$. The $\SL_n(T)$ orbits on pairs symmetric bilinear forms $(A,B) \in V(T)$ with $f_{(A,B)} = f$ are in bijection with equivalence classes of triples $$(I,s,\delta)$$ where $I \subset K_f$ is a fractional ideal of $R_f$, $s$ is in the fraction field of $T$ and is such that $N(I) = sT$, and $\delta \in K_f^\times$ such that $I^2 \subset \delta R_f$ as ideals and the norm equation $s^2 = N(\delta)$ holds. Two triples $(I,s,\delta)$ and $(I',s',\delta')$ are equivalent if there exists a $\kappa \in K_f^\times$ with the property that $I = \kappa I'$, $s = N(\kappa) s'$, and $\delta = \kappa^2 \delta'$.
\end{thm}

Then, we have the following theorem whose proof is straightforward.

\begin{lem}
The stabiliser in $\SL_n(T)$ corresponds to the $2$-torsion of $R_f^\times$ which have norm $1$: $$R_f^\times[2]_{N \equiv 1}.$$ 
\end{lem}

\begin{rem}
It follows that the $\SL_n(\mathbb{Q})$ stabiliser of any element $v$ whose resolvent is irreducible is equal to $2$. 
\end{rem}

Just as in \cite{HoShankarVarmaOdd}, we can compute the number of orbits with the caveat that we need to distinguish the case where $f$ has an odd degree factor, from the case where all of $f$'s factors are even. 

\begin{lem}
Let $T$ be a field or $\mathbb{Z}_p$. Let $f$ be a monic separable, non-degenerate binary form in $U_1(T)$. Then the projective $\SL_n(T)$ orbits of $V(T)$ with resolvent $f$ are in bijection with elements of
\begin{equation*}
(R_f^\times/(R_f^\times)^2)_{N \equiv 1} .
\end{equation*}
if $f$ has an odd degree factor and have a $2$ to $1$ map to 
\begin{equation*}
(R_f^\times/(R_f^\times)^2)_{N \equiv 1}
\end{equation*}
if $f$ only has even degree factors. 
\end{lem}

\begin{rem} 
We can describe the real orbits. Suppose that $f$ is a non-degenerate monic polynomial of degree $n$ with $r_1$ real roots and $2r_2$ complex roots.  If $r_1 = 0$, then there are $2$ $\SL_n(\mathbb{R})$ orbits in $V(\mathbb{R})$ with resolvent polynomial $f$. If $r_1 > 1$, there are $2^{r_1-1}$ $\SL_n(\mathbb{R})$ orbits in $V(\mathbb{R})$ with resolvent polynomial $f$. Furthermore, for $r_1 = 0$, the stabiliser has size $2^{r_2}$ while for $r_1 > 0$ it has size $2^{r_1+r_2-1}$. The size of the stabilizer depends only on $r_2$ and we denote it by $\sigma(r_2)$. 
\end{rem} 

\subsection{The counting problem} We first count the average number of $2$-torsion elements in the oriented class group of monogenised rings and fields of even degree. To make sense of this, we order monogenised fields using the naive height on the minimal polynomial of a generator of the ring of integers whose trace is contained in $[0,n)$. Take a monic integral polynomial $f(x) = x^n+a_1x^{n-1}+\ldots+a_n \in \mathbb{Z}[x]$. We define the naive height of $f$ by: $$H(f):= \max\{|a_i|^{1/i}\} = \max \{|a_1|,|a_2|^{1/2},\ldots,|a_n|^{1/n}\}.$$ Note that $H$ has the property that $$H(\lambda B) = \lambda H(f)$$ so that $H$ is homogeneous of degree $1$. This will be needed when we apply arguments from the geometry of numbers. The goal of the paper is to determine the following averages: $$\lim_{X \rightarrow \infty} \frac{\sum\limits_{\substack{\mathcal{O} \in \mathfrak{R} \\ H(\mathcal{O}) < X}} \left| \Cl_2^*(\mathcal{O}) \right| - \left| \mathcal{I}^*_2(\mathcal{O}) \right|}{\sum\limits_{\substack{\mathcal{O} \in \mathfrak{R} \\ H(\mathcal{O}) < X}} 1},$$ $$\lim_{X \rightarrow \infty} \frac{\sum\limits_{\substack{\mathcal{O} \in \mathfrak{R} \\ H(\mathcal{O}) < X}} \left| \Cl_2(\mathcal{O}) \right| - \left| \mathcal{I}_2(\mathcal{O}) \right|}{\sum\limits_{\substack{\mathcal{O} \in \mathfrak{R} \\ H(\mathcal{O}) < X}} 1},$$ and $$\lim_{X \rightarrow \infty} \frac{\sum\limits_{\substack{\mathcal{O} \in \mathfrak{R} \\ H(\mathcal{O}) < X}} \left| \Cl_2^+(\mathcal{O}) \right| - \left| \mathcal{I}_2(\mathcal{O}) \right|}{\sum\limits_{\substack{\mathcal{O} \in \mathfrak{R} \\ H(\mathcal{O}) < X}} 1},\\$$
where $\mathfrak{R} \subset \mathfrak{R}^{r_1,r_2}$ is any acceptable family of monogenic rings (an acceptable family is one which includes all rings with squarefree discriminant). The asymptotic formula for the denominator is found in the work of Bhargava--Shankar--Wang, \cite{BhargavaShankarWangSquarefreeI}. 

\begin{thm}[Bhargava--Shankar--Wang, \cite{BhargavaShankarWangSquarefreeI}]
Let $S = (S_p)$ be an acceptable collection of local specifications. If $ 0\le b < n$ is fixed and $U_{1,b}$ denotes the set of monic polynomial whose $x^{n-1}$ coefficient is $b$, then we have $$\left| U_{1,b}^{r_2}(S)_{< X}^{{\rm irr}} \right| = \Vol(U_{1,b}^{r_2}(\mathbb{R})_{< X}) \prod_p \Vol(S_p) + o(X^{\frac{n(n+1)}{2}-1}).$$
\end{thm}

Now $\Vol(S_{\infty,H<X})$ grows like $X^{\frac{n(n+1)}{2}-1}$. Thus, the main term dominates the error term.

\section{Reduction theory} \label{Reduction theory}

Fix an element $A \in \mathscr{L}_{\mathbb{Z}}$ and $\delta \in \mathcal{T}(r_2)$. We build a finite cover of the fundamental domain for the action of $\SO_A(\mathbb{Z})$ on $V_A^{r_2,\delta}(\mathbb{R})$. 

\begin{defn}
The height of an element in $B \in V_A^{r_2,\delta}$ is defined to be the height of the associated resolvent polynomial. That is, $$H(B):= H({\rm disc}(Ax-By)) = H \left((-1)^{\frac{n}{2}}\det(Ax-B)\right).$$
\end{defn}

The construction of \cite{BhargavaPointless} can be adapted to give a fundamental set $R_{A}^{r_2,\delta}$ for the action of $\SO_A(\mathbb{R})$ on $V_A^{r_2,\delta}(\mathbb{R})$ (which could be empty) with the following properties: 

\begin{enumerate}
\item The set $R_{A}^{r_2,\delta}$ is a semi-algebraic.
\item If $R_{A}^{r_2,\delta}(X)$ denotes the set of elements of height at most $X$, then the coefficients of elements $B \in R_{A}^{r_2,\delta}(X)$ are bounded by $O(X)$. The implied constant is independent of $B$.
\end{enumerate}

We define an indicator function that records whether $V_A^{r_2,\delta}(\mathbb{R})$ is empty.  

\begin{defn}
We define the indicator function
$$\chi_A(\delta):=
\begin{cases}
1 &\mbox{if } V_A^{r_2,\delta}(\mathbb{R}) \neq \emptyset\\
0 &\mbox{otherwise }
\end{cases}.$$ 
\end{defn}

We can build now build a cover of a fundamental domain for the action of $\SO_A(\mathbb{Z})$ on $V_A^{r_2,\delta}(\mathbb{R})$. To do so, we pick a fundamental domain $\mathcal{F}_A$ for the action of $\SO_A(\mathbb{Z})$ on $\SO_A(\mathbb{R})$ and act on $R_{A}^{r_2,\delta}$. This gives a $\frac{\sigma(r_2)}{2}$ cover of a fundamental domain for the action of $\SO_A(\mathbb{Z})$ where $\sigma(r_2)$ is the size of the stabiliser in $\SO_A(\mathbb{R})$ of an element $v \in V_A^{r_2,\delta}(\mathbb{R})$. 

\begin{prop}
Let $\mathcal{F}_A$ be a fundamental domain for the action of $\SO_A(\mathbb{Z})$ on $\SO_A(\mathbb{R})$. Then
\begin{enumerate}
\item If $\chi_A(\delta) = 1$, $\mathcal{F}_A \cdot R_{A}^{r_2,\delta}$ is an $\frac{\sigma(r_2)}{2}$--fold cover of a fundamental domain for the action of $\SO_A(\mathbb{Z})$ on $V_{A}^{r_2,\delta}(\mathbb{R})$, where we regard $\mathcal{F}_A \cdot R_{A}^{r_2,\delta}$ as a multiset.
\item If $\chi_A(\delta) = 0$, then $\emptyset$ is a fundamental domain.
\end{enumerate} 
\end{prop}
\begin{proof}
The stabiliser in $\SO_A(\mathbb{R})$ of an element $B \in V_{A}^{r_2,\delta}(\mathbb{R})$ coincides with the stabiliser in $\SL_n(\mathbb{R})$ of $(A,B)$ which has size $\sigma(r_2)$. The factor of $\frac{1}{2}$ comes from the fact that $-1$ also always stabilises $(A,B)$. 
\end{proof}

\begin{rem}
The characteristic functions will be used to define the archimedean mass and will make the final computation more transparent. 
\end{rem}

\section{Averaging and cutting off the cusp} \label{Averaging and cutting off the cusp}

For the purpose of cutting off the cusp and averaging, it suffices to work with $\SO_A$ instead of $\SO_A$ since $\SO_A(\mathbb{Z}) \backslash \SO_A(\mathbb{R})$ is in bijection with $\SO_A(\mathbb{Z}) \backslash \SO_A(\mathbb{R})$. \\

There are now two different cases to consider: 1) the case where $A$ is anisotropic over $\mathbb{Q}$ and 2) the case where $A$ is isotropic over $\mathbb{Q}$. For each, we need to show that: 
\begin{enumerate}
\item the number of absolutely irreducible integral points in the cuspidal region is negligible; and 
\item the number of reducible integral points in the main body is negligible. \\
\end{enumerate}

We define absolutely irreducible points and reducible points and set the notation for the remainder of this section. 

\begin{defn}
An element $v \in V(\mathbb{Z})$ is said to be absolutely irreducible if $v$ does not correspond to the identity element in the class group and the resolvent of $v$ corresponds to an order in an $S_n$-field. An element which is not absolutely irreducible is said to be reducible. 
\end{defn}

We have the following theorem which gives conditions on reducibility.

\begin{thm}[Reducibility criterion]
Let $(A, B) \in V(\mathbb{Z})$ be such that all the variables in one of the following sets vanish. Then $(A,B)$ is reducible. 
\begin{enumerate}
\item {\bf \emph{The modified squares:}} $$\{ a_{i,j}, b_{i,j}\}$$ for all $1 \le i,j \le \frac{n}{2}$ except for $i = j = \frac{n}{2}$ where $a_{\frac{n}{2}\, \frac{n}{2}} = 0$. \\ These pairs correspond to the identity element in the class group. 
\item {\bf \emph{The rectangles:}} $$\left\{a_{ij},b_{ij} | 1 \le i \le k,  1 \le j \le n-k\right\}$$ for some $1 \le k \le n-1$. \\ These pairs correspond to the resolvent having repeated roots. 
\end{enumerate}
\end{thm}

We define the affine spaces $V_{A,b}$ just as we did in the odd degree case \cite{SiadOddMonogenicAverages}.

\begin{defn}
Let $A$ be a fixed quadratic form in $\mathscr{L}_{\mathbb{Z}}$ and fix $0 \le b < n$. We let $V_A \subset V$ denote the space of pairs $(A,B)$, where $B$ is arbitrary.  Note that the resolvent map takes $V_A$ to $U$. Now, we let $V_{A,b}$ denote the inverse image under the resolvent map of the set $U_b$. It is easy to see that $V_{A,b}$ is an affine subspace of $V_A$ of dimension $\frac{n(n+1)}{2}-1$.
\end{defn}

\begin{defn}
Let $S \subset V_{A,b}^{r_2,\delta}(\mathbb{Z}) := V_{A,b}^{r_2,\delta} (\mathbb{\mathbb{R}})\cap V_{A,b}(\mathbb{Z})$ be an $\SO_A(\mathbb{Z})$ invariant set. Denote by $N_H(S;X)$ the number of absolutely irreducible $\SO_A(\mathbb{Z})$-orbits on $S$ that have height bounded by $X$. For any $L \subset V_A(\mathbb{Z})$, let $L^{{\rm irr}}$ denote the set of absolutely irreducible elements. Note that any absolutely irreducible element has a resolvent form corresponding to an order $\mathcal{O}$ in an $S_n$-number field and so $\mathcal{O}^\times[2]$ has size $2$. As a result, the stabiliser in $\SO_A(\mathbb{Z})$ of absolutely irreducible elements has size $2$. 
\end{defn}

Therefore, we have $$N_H(S;X) = \frac{2}{\sigma(r_2)} \#\{\mathcal{F}_A \cdot R_A^{r_2,\delta}(X) \cap S^{{\rm irr }}\}.$$

The goal of this section is to obtain an asymptotic formula for $N_H(S;X)$.

\subsection{The case of $A$ anisotropic over $\mathbb{Q}$}  When $A$ is anisotropic, we can pick a compact fundamental domain $\mathcal{F}_A$ for the action of $\SO_A(\mathbb{Z})$ on $\SO_A(\mathbb{R})$. It then follows that $\mathcal{F}_A \cdot R_{A,b}^{r_2,\delta}$ is bounded. To estimate the number of absolutely irreducible integral points in the fundamental domain for the action of $\SO_A(\mathbb{Z})$ on $V_{A,b}^{r_2,\delta}$, we apply results from the geometry of numbers directly. We use Davenport's refinement of the Lipschitz method on $\mathcal{F}_A \cdot R_{A,b}^{r_2,\delta}(X)$ to obtain the desired asymptotic formula. 

We will need the following version of Davenport's lemma. 

\begin{lem}[Davenport's Lemma]
Let $E \subset \mathbb{R}^n$ be a bounded semi-algebraic multiset with maximum multiplicity at most $m$ which is defined by $k$ algebraic inequalities of each having degree at most $l$. Let $E'$ be the image of $E$ under any upper/lower triangular unipotent transformation. Then the number of integral points in $E'$ counted with multiplicity is $$ \Vol(E)+O_{m,k,l}\left(\max \{\Vol(\overline{E}),1\}\right)$$ where $\Vol(\overline{E})$ denotes the greatest $d$-dimensional volume of a projection of $E$ onto a $d$-dimensional coordinate hyperplane for $1 \le d \le n-1$. 
\end{lem}

\begin{lem}
The number of integral points in $\mathcal{F}_A \cdot R_{A,b}^{r_2,\delta}$ which are not absolutely irreducible is bounded by $o\left(X^{\frac{n(n+1)}{2}-1}\right)$. 
\end{lem}
\begin{proof} The proof is the same as in \cite{SiadOddMonogenicAverages}, and follows directly from adapting the results of Ho--Shankar--Varma \cite{HoShankarVarmaOdd}. 
\end{proof}

We thus obtain the following asymptotic formula for $N_H(S;X)$. 

\begin{thm}
Let $A \in \mathscr{L}_{\mathbb{Z}}$ be anisotropic over $\mathbb{Q}$. We have $$N(V_{A,b}^{r_2,\delta}(\mathbb{Z});X) = \frac{1}{\sigma(r_2)} \Vol \left(\mathcal{F}_A \cdot R_{A,b}^{r_2,\delta}(X)\right) + o(X^{\frac{n(n+1)}{2}-1}).$$
\end{thm}

\subsection{The case of $A$ isotropic over $\mathbb{Q}$} 

The arguments are almost the same as in the odd degree case \cite{SiadOddMonogenicAverages} except for the totally split case where the Iwasawa coordinates change slightly. 

Suppose now that $A$ is isotropic over $\mathbb{Q}$. Then there exists an element $g_A \in \SL_n(\mathbb{Q})$ such that $g_A^tAg_A = A_{F_0}$ where \begin{equation*} 
A_{F_0} := 
\begin{pmatrix} 
& & & & & & 1 \\ 
& & & & & \iddots & \\ 
& & & & 1 & & \\ 
& & & F_0 & & & \\  
& & 1 & & & &  \\ 
& \iddots & & & & & \\ 
1 & & & & & &  
\end{pmatrix}. 
\end{equation*} 
where $F_0$ is a $\mathbb{Q}$ anisotropic form. We define $m = \frac{n - \dim(F_0)}{2}$. We note that $0 < m \le \frac{n}{2}$. 

Now for $K = \mathbb{Q}$ or $\mathbb{R}$, we consider the maps
\begin{align*}
\sigma_V &\colon V_{A,b}^{r_2,\delta} \rightarrow V_{A_{F_0},b}^{r_2,\delta} \\
\sigma_A &\colon \SO_A(K) \rightarrow \SO_{A_{F_0}}(K)
\end{align*}
defined by $\sigma_V(A,B) = (A_{F_0},g_A^tBg_A)$ and $ \sigma_A(h) = g_A^t h (g_A^t)^{-1}$. We note that $$H(A,B) = H(\sigma_V(A,B))$$ since $\pi \circ \sigma_V = \pi$. Furthermore, $\sigma_V(h \cdot v) = \sigma_A(h)\cdot \sigma_V(v)$.

Now, we denote by $\mathcal{L} \subset V_{A_{F_0},b}^{r_2,\delta}(\mathbb{R})$ the lattice $\sigma_V\left(V_{A_{F_0},b}^{r_2,\delta}(\mathbb{Z})\right)$. We denote by $\Gamma \subset \SO_{A_{F_0}}(\mathbb{R})$ the subgroup $\sigma_A(\SO_{A}(\mathbb{Z}))$. This subgroup is commensurable with $\SO_{A_{F_0}}(\mathbb{Z})$. Therefore, there exists a fundamental domain $\mathcal{F}$ for the action of $\Gamma$ on $\SO_{A_{F_0}}(\mathbb{R})$ which is contained in a finite union of $\SO_{A_{F_0}}(\mathbb{Q})$ translates of a Siegel domain, $\bigcup_i g_i \mathcal{S}$ for $g_i \in \SO_{A_{F_0}}(\mathbb{Q})$. This is known from \cite{BorelHarishChandra}. 

The choice of the standard $A_{F_0}$ as above is convenient at this point. Indeed, we may now choose as our Siegel domain $\mathcal{S}$ the product $NTK$ where we choose $K$ to be compact, $N$ to be a subgroup of the group of lower triangular matrices with $1$ on the diagonal and $T$ to be $$T := \left\{\begin{pmatrix} 
t_1^{-1} & & & & & & \\ 
& \ddots & & & & & \\ 
& & t_m^{-1} & & & & \\  
& & & I_{\dim(F_0)} & & & \\ 
& & & & t_m & & \\ 
& & & & & \ddots & \\ 
& & & & & & t_1 \\  
\end{pmatrix} \colon t_1/t_2 > c, \ldots, t_{m-1}/t_m > c, t_m >c \right\}$$ for some constant $c>0$ if $\dim(F_0) > 0$ and $$T_{0} := \left\{\begin{pmatrix} 
t_1^{-1} & & & & &  \\ 
& \ddots & & & &  \\ 
& & t_m^{-1} & & &  \\  
& & &  t_m & & \\ 
& & &  & \ddots & \\ 
& & & & & t_1 \\  
\end{pmatrix} \colon t_1/t_2 > c, \ldots, t_{m-1}/t_m > c, t_{m-1}t_{m} >c \right\}$$ for some constant $c>0$ if $\dim(F_0) = 0$. This can be found in many sources, see for instance \cite{BorelEnsemblesFondamentaux}, \cite{SawyerIwasawa}, or \cite{SawyerSpherical}. 

When $m < \frac{n}{2}$, note that $s_i = t_i/t_{i+1}$, $0 \le i \le m-1$ and $s_m = t_m$ forms a set of simple roots.

When $m = \frac{n}{2}$, note that $s_i = t_i/t_{i+1}$, $0 \le i \le m-1$ and $s_m = t_{m-1}t_m$ forms a set of simple roots.

Moreover, if we denote by $e^\rho$ the exponential of the half sum of the positive roots counted with multiplicities, we have 
\begin{equation*}
e^{\rho} = \prod_{i=1}^{m} t_i^{\frac{n}{2}-i}.
\end{equation*} 

When $m < \frac{n}{2}$, we find:
\begin{align*}
e^{\rho} &= \prod_{i=1}^{m} t_i^{\frac{n}{2}-i} \\
&= \prod_{i=1}^{m} \left(\prod_{j=i}^{m} s_j \right)^{\frac{n}{2}-i} \\
&= \prod_{i=1}^{m} s_i^{\left(\sum_{j=1}^{i}\frac{n}{2}-j\right)} \\
&= \prod_{i=1}^{m} s_i^{i \left(\frac{n-i-1}{2}\right)}.
\end{align*}

When $m = \frac{n}{2}$, we find: 
\begin{align*}
e^{\rho} &= \prod_{i=1}^{m} t_i^{\frac{n}{2}-i} \\
&=\prod_{i=1}^{m-2} \left((s_{m-1}s_{m})^{-\frac{1}{2}}\prod_{j=i}^{m} s_j \right)^{\frac{n}{2}-i}  (s_{m-1}s_{m})^{\frac{1}{2}} \\
&= \prod_{i=1}^{m-2} s_i^{i \left(\frac{n-i-1}{2}\right)} (s_{m-1}s_{m})^{\frac{1}{2}} (s_{m-1}s_{m})^{\frac{1}{2}(m-2)(\frac{n}{2}-\frac{m-1}{2})} \\
&=  \prod_{i=1}^{m-2} s_i^{i \left(\frac{n-i-1}{2}\right)} (s_{m-1}s_{m})^{\frac{n(n-2)}{16}} .
\end{align*}

We now fix some notation for our choice of Haar measure on $G = \SO_{A_{F_0}}$. We let $dg$ denote the Haar measure on $G$, $dn$ denote the Haar measure on the unipotent group $N$, and $dk$ denote the Haar measure on the compact group $K$. For every $1 \le i \le m$ we write $d^{\times}t_i = \frac{dt_i}{t_i}$ and $d^{\times}s_i = \frac{ds_i}{s_i}$. Furthermore, we write $dt =\prod_{i=1}^m dt_i$, $d^{\times}t = \prod_{i=1}^m d^{\times}t_i$ and $ds = \prod_{i=1}^{m}ds_i$, $d^{\times}s = \prod_{i=1}^{m} d^{\times}s_i$. 

Changing variables between the $t$-coordinates and the $s$-coordinates gives us  $$d^\times t = d^\times s.$$

Therefore, if $m < \frac{n}{2}$, the Haar measure is given in $NTK$-coordinates by 
\begin{align*}
dg &= e^{-2\rho} du \, d^{\times}t \, dk \\
&= \prod_{i=1}^{m} t_i^{2i-n} du \, d^{\times}t \, dk \\
&= \prod_{i=1}^{m} s_i^{i(i+1-n)} du \, d^{\times}s \, dk.
\end{align*} 

Therefore, if $m = \frac{n}{2}$, the Haar measure is given in $NTK$-coordinates by 
\begin{align*}
dg &= e^{-2\rho} du \, d^{\times}t \, dk \\
&= \prod_{i=1}^{m} t_i^{2i-n} du \, d^{\times}t \, dk \\
&=  \prod_{i=1}^{m-2} s_i^{i \left(i+1-n\right)} (s_{m-1}s_{m})^{-\frac{n(n-2)}{8}} du \, d^{\times}s \, dk.
\end{align*} 

We define the main body and the cuspidal region of the multiset $\mathcal{F}_A \cdot R_{A,b}^{r_2,\delta}$. 

\begin{defn}[Main body and cuspidal region] The \emph{main body} consists of all the elements of $\mathcal{F}_A \cdot R_{A,b}^{r_2,\delta}$ for which $|b_{11}| \ge 1$. The \emph{cuspidal region} consists of all the elements for which $|b_{11}| < 1$.

\end{defn}

We are now ready to cut off the cuspidal region. 

\begin{con}[Partial order on the coordinates of $V_A$]

We construct a partial order on the $n(n+1)/2$ coefficients $\{b_{ij}\}$ for $i \le j$. These define a set of coordinates on $B$ which we denote by $U$. 

\begin{defn}
The weight $w(b_{ij})$ of an element $b_{ij} \in U$ is the factor by which $b_{ij}$ scales under the action of $(t_1^{-1},\ldots,t_m^{-1},1, \ldots,1,t_m,\ldots,t_1) \in T$.  
\end{defn}

We are now ready to define a partial order on $U$. 

\begin{defn}[A partial order on subsets of $U$] Let $b$ and $b'$ be two elements of the set of coordinates $U$. We say that $b \prec b'$ if in the expression for $w(b)$ in the $s$-coordinates, the exponents of the variables $s_1,\cdots,s_m$ are smaller than or equal to the corresponding exponents appearing in the expression for $w(b')$ in the $s$-coordinates. The relation $\prec$ defines a partial order on $U$. 
\end{defn} 

\begin{exmp}
We have $b_{11} \prec b_{m+1\,m+1}$ because $w(b_{11}) = s_1^{-2} \cdots s_m^{-2}$ while $w(b_{m+1\,m+1}) = 1 = s_1^0 \cdots s_m^0$. On the other hand, $b_{1\,n-2}$ and $b_{2\,n-3}$ cannot be compared in $\prec$ because $w(b_{1\,n-2}) = s_1^{-1} s_2^{-1}$ while $w(b_{2\,n-3}) = s_2^{-1} s_3^{-1}$. The important thing to note about the partial order $(U,\prec)$ is that if $i \le i'$ and $j \le j'$ then $$b_{ij} \prec b_{i'j'}.$$  
\end{exmp}
\end{con}

We now cut off the cusp in two specific cases which will serve as bases cases in the proof by induction of the general case.

\begin{exmp}[Base case of cusp cutting induction for $\dim(F_0)=0$]
We now do the case $n=4$, $m=2$ before moving on to cutting off the cusp in the general case. We see that torus elements act as follows: 
$$t\cdot v = \begin{pmatrix}
t_1^{-2} & t_1^{-1}t_2^{-1} & t_1^{-1} t_2 & 1 \\
t_1^{-1}t_2^{-1} & t_2^{-2}  & 1 & t_1 t_2^{-1}\\
t_1^{-1} t_2 & 1 & t_2^2 & t_1t_2 \\ 
1 & t_1 t_2^{-1} & t_1t_2 & t_1^2
\end{pmatrix} O(X).$$

We can now easily read off the weights. 
The Haar measure takes the form $$dg = du \, \frac{1}{t_1^2} d^{\times}t \, dk =  du \, \frac{1}{s_1 s_2} d^{\times}s \, dk .$$

For any subset of $U$ containing $b_{11}$, we now want to estimate $$\widetilde{I}(U_1,X) = X^{9-\#U_1}  \int_{t \in T_X} \prod_{b_{ij} \not\in U_1} w(b_{ij}) \frac{d^{\times}s}{s_1 s_2}.$$ 

We only need to look at proper subsets of $U_0 = \{b_{11}\}$ which are left-closed and up-closed. Recall that we have the bound $s_1< CX$ and $s_2 < C^2X^2$. Let's compute: 
\begin{alignat*}{2}
&\widetilde{I}(\{b_{11}\},X) &&= X^{8} \int_{s_1=c}^{CX} \int_{s_2 = c}^{C^2X^2} s_1s_2 \frac{d^{\times}s}{s_1 s_2} = X^{8} \int_{s_1=c}^{CX} \int_{s_2 = c}^{C^2X^2} d^{\times}s = O_\epsilon(X^{8+\epsilon}) \\
\end{alignat*}

The number of absolutely irreducible elements in the cusp which have height at most $X$ is thus $O_\epsilon(X^{9-1+\epsilon})$ and we just barely cut off the cusp! The induction argument given below shows that in all other cases, we have much more room!
\end{exmp}

We recall that we had $$N_H(S;X) = \frac{1}{\sigma(r_2)} \#\{\mathcal{F}_A \cdot R_A^{r_2,\delta}(X) \cap S^{{\rm irr }}\}.$$

Now, let $G_0$ be a bounded open $K$-invariant ball in $\SO_{A_{F_0}}(\mathbb{R})$. We can average the above expression by the usual trick to obtain $$N_H(S;X) = \frac{1}{\sigma(r_2) \Vol(G_0)} \int_{h \in \mathcal{F}_A} \# \left\{ hG_0 R_A^{r_2,\delta}(X) \cap S^{{\rm irr}} \right\} dh.$$

Now, again we may use classical arguments to see that the number of absolutely irreducible integral points in the cusp which have height at most $X$ is $$O\left(\int_{t \in T} \# \left\{tG_0 R_A^{r_2,\delta}(X) \cap S^{{\rm irr}} \right\} \prod_{i=1}^{m} s_i^{i(i+1-n)} d^\times s \right)$$ when $m < \frac{n}{2}$ and $$O\left(\int_{t \in T} \# \left\{tG_0 R_A^{r_2,\delta}(X) \cap S^{{\rm irr}} \right\} \prod_{i=1}^{m} s_i^{i(i+1-n)} (s_{m-1}s_{m})^{-\frac{n(n-2)}{8}} d^\times s \right)$$
when $m = \frac{n}{2}$.

\begin{defn}
Let $U_1 \subset U$ be a subset of the set of coordinates. We define $$V_A(\mathbb{R})(U_1) = \left\{B \in V_A(\mathbb{R}) \colon \left| b_{ij}(B) \right| < 1 \text{ if and only if } b_{ij} \in U_1 \right\}$$ and $$V_A(\mathbb{Z})(U_1) = V_A(\mathbb{Z}) \cap V_A(\mathbb{R})(U_1).$$
\end{defn}

It thus suffices to show that $$N(V_A(\mathbb{Z})(U_1);X) = O_\epsilon\left(X^{\left(\frac{n(n+1)}{2}-1\right)-1+\epsilon}\right)$$ for all $U_1 \subset U$ such that $b_{11} \in U_1$. 

We get a priori bounds on the coordinates $s_i$ from the reducibility criterion. Let $C$ be an absolute constant such that $CX$ bounds the absolute value of all the coordinates of elements $B \in G_0 R_A^{r_2,\delta}(X)$. 

If $(s_1^{-1},\ldots,s_m^{-1},1,\ldots,1,s_m,\ldots,s_1) \in T$ and $CX w(b_{i_0 \, n-i_0}) < 1$ for some $i_0 \in \{1,\ldots,m\}$, then $CXw(b_{ij})<1$ for all $i \le i_0$ and $j \le n-i_0$. This comes from the {\bf \emph{Rectangles}} part of the criterion for reducibility. Therefore, we may assume that $$s_i < CX$$ for all $i \in \{1,\ldots,m\}$ if $m< \frac{n}{2}$ and that $$s_i < CX$$ for all $i \in \{1,\ldots,m-1\}$ and $s_m < C^2 X^2$ if $m = \frac{n}{2}$.

Let us write $T_X$ to denote the set of $t = (s_1,\ldots,s_m) \in T$ which satisfy this condition.  

Now Davenport's lemma gives us 
\begin{align*}
N(V(\mathbb{Z})(U_1);X) &= O\left(\int_{t \in T_X} \Vol(t G_0 R_A^{r_2,\delta}(X) \cap V(\mathbb{R})(U_1)) \prod_{i=1}^{m} s_i^{i(i+1-n)} d^{\times}s  \right) \\
&= O\left(X^{\left(\frac{n(n+1)}{2}-1\right)-\#U_1} \int_{t \in T_X} \prod_{b_{ij} \not\in U_1} w(b_{ij}) \prod_{i=1}^{m} s_i^{i(i+1-n)} d^\times s \right).
\end{align*}
for $m < \frac{n}{2}$ and 
\begin{align*}
N(V(\mathbb{Z})(U_1);X) &= O\left(\int_{t \in T_X} \Vol(t G_0 R_A^{r_2,\delta}(X) \cap V(\mathbb{R})(U_1)) \prod_{i=1}^{m} s_i^{i(i+1-n)} d^{\times}s  \right) \\
&= O\left(X^{\left(\frac{n(n+1)}{2}-1\right)-\#U_1} \int_{t \in T_X} \prod_{b_{ij} \not\in U_1} w(b_{ij}) \prod_{i=1}^{m} s_i^{i(i+1-n)} (s_{m-1}s_{m})^{-\frac{n(n-2)}{8}} d^\times s \right).
\end{align*}
for $m = \frac{n}{2}$.

So, we have reduced our problem to one of estimating the following integrals. 

\begin{defn}
The active integral of $U_1 \subset U$ is defined by $$\widetilde{I}(U_1,X) := X^{\left(\frac{n(n+1)}{2}-1\right)-\#U_1}  \int_{t \in T_X} \prod_{b_{ij} \not\in U_1} w(b_{ij}) \prod_{i=1}^{m} s_i^{i(i+1-n)} d^\times s $$ if $m < \frac{n}{2}$ and 
$$\widetilde{I}(U_1,X) := X^{\left(\frac{n(n+1)}{2}-1\right)-\#U_1}  \int_{t \in T_X} \prod_{b_{ij} \not\in U_1} w(b_{ij})\prod_{i=1}^{m} s_i^{i(i+1-n)} (s_{m-1}s_{m})^{-\frac{n(n-2)}{8}} d^\times s $$
if $m = \frac{n}{2}$. 
\end{defn}

Recall, that $b_{ij} \prec b_{i_0j_0}$ when $i \le i_0$ and $j \le j_0$. Therefore, if $U_1 \subset U$ contains $b_{i_0j_0}$ but not $b_{ij}$, then $$\widetilde{I}\left(U_1 \setminus \{b_{i_0j_0} \} \cup \{b_{ij}\}, X \right) \ge \widetilde{I}(U_1,X).$$ As a result, in order to obtain an upper bound for $\widetilde{I}(U_1,X)$ we may assume that if $b_{i_0j_0} \in U_1$, then $b_{ij} \in U_1$ for all $i \le i_0$ and $j \le j_0$. In other words, we may assume that $U_1$ is both left closed and up closed. 

Furthermore, such a set $U_1$ cannot contain any element on, or on the right of, the off anti-diagonal within the first $m$-rows, since otherwise, we would be in the case of {\bf \emph{Rectangles}} in the reducibility criterion and so $N(V(\mathbb{Z})(U_1);X)=0$. 

\begin{defn}
We define the subset $U_0 \subset U$ as the set of coordinates $b_{ij}$ such that $i \le j$, $i \le m$, and $i+j \le n-1$. 
\end{defn}

Now, if $m = \frac{n}{2}$, every element in $V(\mathbb{Z})(U_0)$ is reducible and it suffices to consider $\widetilde{I}(U_1,X)$ for all $U_1 \subsetneq U_0$. On the other hand if $m < \frac{n}{2}$ we need to consider all $U_1 \subset U$.

Since the product of the weight over all the coordinates is $1$, we make the following definition.

\begin{defn}
We define for a subset $U_1 \subset U$ $$I(U_1,X) = X^{\frac{n(n+1)}{2}-1} \widetilde{I}(U_1,X) = X^{-\#U_1}  \int_{t \in T_X} \prod_{b_{ij} \in U_1} w(b_{ij})^{-1} \prod_{i=1}^{m} s_i^{i(i+1-n)} d^\times s $$ if $m < \frac{n}{2}$ and $$I(U_1,X) = X^{\frac{n(n+1)}{2}-1} \widetilde{I}(U_1,X) = X^{-\#U_1}  \int_{t \in T_X} \prod_{b_{ij} \in U_1} w(b_{ij})^{-1} \prod_{i=1}^{m} s_i^{i(i+1-n)} (s_{m-1}s_{m})^{-\frac{n(n-2)}{8}} d^\times s $$ if $m = \frac{n}{2}$. 
\end{defn}

We are now ready to state and prove the main cusp cutting lemma. 

\begin{lem}[Main cusp cutting estimate]
Let $U_1$ be a non-empty proper subset of $U_0$. Then we have the estimate $$I(U_1,X) = O_\epsilon \left(X^{-1+\epsilon}\right).$$
We also have $I(\emptyset) = O_\epsilon\left(X^{\frac{1}{6}m(m+1)(2m-3n+4)+\epsilon}\right)$ and $I(U_0) = O_\epsilon\left(X^{m(2m+1-n)+\epsilon}\right)$. 
\end{lem}

\begin{proof}
We prove this lemma via a combinatorial argument using induction on $m$. Recall that $n=2m+\dim(F_0)$. The cases $\dim(F_0) \ge 2$ and $\dim(F_0) = 0$ are slightly different and we handle them separately. 

The case $\dim(F_0) = 0$ is actually the same as is \cite{ShankarWang2Selmer} with a different normalization of the height. It thus suffices to consider the case $\dim(F_0) \ge 2$.

To start, let us assume that $\dim(F_0) > 0$. First, we compute $I(U_0,X)$ 
\begin{align*}
I(U_0,X) &= X^{-\#U_0} \int_{t \in T_X} \prod_{b_{ij} \in U_0} w(b_{ij})^{-1} \prod_{i=1}^{m} s_i^{i(i+1-n)} d^{\times}s. \\ 
&= X^{-m(n-(m+1))} \int_{t \in T_X} \left(t_1^{n-2+1} t_2^{n-4+2} t_3^{n-6+2} \cdots t_m^{n-2m+2} \right)   \prod_{i=1}^{m} t_i^{2i-n} d^{\times}t \\
&= X^{-m(n-(m+1))} \int_{t \in T_X} t_1t_2^2\cdots t_m^2 d^{\times}t \\
&= X^{-m(n-(m+1))} \int_{s_1,\ldots,s_m=c}^{CX} s_1 s_2^3 s_3^5 \cdots s_m^{2m-1} d^{\times}s \\
&= O\left(X^{-m(n-(m+1))+m^2}\right) \\
&= O\left(X^{m(2m+1-n)}\right) \\
&= O\left(X^{-m(\dim(F_0)-1)}\right).
\end{align*}

We compute $I(\emptyset,X)$ directly
\begin{align*}
 I(\emptyset,X) =\int_{s_1, \ldots, s_n = c}^{CX} \prod_{i=1}^{m} s_i^{i(i+1-n)} d^{\times}s &= O\left(X^{\frac{1}{6}m(m+1)(2m-3n+4)}\right) \\
&= O\left(X^{\frac{1}{6}m(m+1)(-4m-3\dim(F_0)+4)}\right).
\end{align*}
Now, let $U_1'$ denote $U_0 \setminus U_1$. Define $I_m'(U_1',X) :=I(U_1,X)$. Then we have: 
\begin{align*}
I_m'(U_1',X) &= X^{\#U_1'-m(n-(m+1))} \int_{t \in T_X} \left( \prod_{b_{ij} \in U_1'} w(b_{ij}) \right) t_1t_2^2\cdots t_m^2 d^{\times}t. \\
&= X^{\#U_1'-m(n-(m+1))} \int_{s_1,\ldots,s_m = c}^{XC} \left( \prod_{b_{ij} \in U_1'} w(b_{ij}) \right) s_1 s_2^3 \cdots s_m^{2m-1} d^{\times}s. 
\end{align*}

We now work out the base case of the induction. When $m = 1$, we have 
\begin{align*}
I_1(\emptyset,X) &= O_\epsilon\left(X^{-\dim(F_0)+\epsilon} \right)  \\
I_1(\{b_{11}\}) &= O_\epsilon \left(X^{1-\dim(F_0)+\epsilon}\right) \\
I_1(\{b_{11}, \ldots, b_{1k} \}) &= O_\epsilon\left( X^{1-\dim(F_0)+\epsilon} \right) \\
I_1(U_0,X) &= O_\epsilon\left(X^{1-\dim(F_0)+\epsilon}\right).
\end{align*}

In particular, when $\dim(F_0) \ge 2$ we see that all these quantities are $O_\epsilon(X^{-1+\epsilon})$. This will come up in the induction step.

Now, suppose that $m \ge 2$. 

Now, for any decomposition $k=k_1+k_2$ we have: $$\int_c^{CX} s^k d^{\times}s \ll_{c,C} \int_{c}^{CX} s^{k_1} d^{\times}s  \int_{c}^{CX} s^{k_2} d^{\times}s.$$ Consequently, we see that $I_n'(U_1',X)$ is bounded by the product $$I_n'(U_1',X) \le J_m(U_2',X) \, K_m(U_3',X),$$ where $U_2'$ consist of all the elements of $U_1'$ in the first row, $U_3'$ consists of the rest of the elements of $U_1'$, and
\begin{alignat*}{2}
&J_m(U_2',X) &&= \left(X^{\#U_2'-(n-2)} \int_{s_1, \ldots, s_n = c}^{CX} \left( \prod_{b_{1j} \in U_2'} w(b_{1j}) \right) s_1 s_2^2 \cdots s_m^2 d^{\times}s \right) \\
&K_m(U_3',X) &&= \left(X^{\#U_3'-\#U_0+(n-2)} \int_{s_2, \ldots, s_n = c}^{CX} \left(\prod_{b_{ij} \in U_3'} w(b_{ij}) \right) s_2 s_3^3 \cdots s_m^{2m} d^{\times}s \right). 
\end{alignat*}

Note that $K_m(U_3',X) = I_{m-1}(U_3',X)$ and we can estimate it by induction. Now, $U_1$ is left-closed and non-empty and hence the subset $U_2'$ is either empty or of the form $\{b_{1\,k}, \ldots, b_{1\,n-2}\}$ for some $k \ge 2$. 

Now, if $U_2' =\emptyset$: \\
\begin{equation*}
J_m(U_2',X) = O_\epsilon\left(X^{2m-1-n+2}\right) = O_\epsilon\left(X^{1-\dim(F_0)+\epsilon} \right) = O_\epsilon\left(X^{-1+\epsilon}\right).
\end{equation*} \\

Now, if $k=2$, then: \\
\begin{align*}
J_m(U_2',X) &= X^{-1} \int_{s_1, \ldots, s_n = c}^{CX}  t_1^{3-n} t_2^{-1} s_1 s_2^2 \cdots s_m^2 d^{\times}s \\
&= O_\epsilon\left(X^{-1+m(3-n)-(m-1) + (2m-1)+\epsilon}\right) \\
&= O_\epsilon\left(X^{-1+3m-m(2m+\dim(F_0))+m+\epsilon}\right) \\
&= O_\epsilon\left(X^{-1+4m-2m^2-m\dim(F_0)+\epsilon}\right) \\
&= O_\epsilon\left(X^{-1+\epsilon}\right).
\end{align*} \\

Now, if $k=3$, then: \\
\begin{align*}
J_m(U_2',X) &= X^{-2} \int_{s_1, \ldots, s_n = c}^{CX}  t_1^{4-n} s_1 s_2^2 \cdots s_m^2 d^{\times}s \\
&= O_\epsilon \left(X^{-2+m(4-n)+2m-1)+\epsilon}\right) \\
&= O_\epsilon\left(X^{-3+6m-2m^2-m\dim(F_0)+\epsilon}\right) \\
&= O_\epsilon \left(X^{-1+\epsilon}\right).
\end{align*} \\

If $4 \le k \le m$, then: \\
\begin{align*}
J_m(U_2',X) &= X^{1-k} \int_{s_1, \ldots, s_n = c}^{CX} t_1^{-((n-2)-k+1)} t_k^{-1} \cdots t_m^{-1} t_m \cdots t_3 s_1 s_2^2 \cdots s_m^2 d^{\times}s \\
&= X^{1-k} \int_{s_1, \ldots, s_n = c}^{CX}  t_1^{-((n-2)-k+1)} t_3 \cdots t_{k-1} s_1 s_2^2 \cdots s_m^2 d^{\times}s \\
&= X^{1-k}  \int_{s_1, \ldots, s_n = c}^{CX}  t_1^{-((n-2)-k+1)} s_3 s_4^2 \cdots s_{k-1}^{k-3} s_1 s_2^2 \cdots s_m^2 d^{\times}s \\
&= O_\epsilon\left(X^{1-k+\frac{(k-3)(k-2)}{2}+2m-1+m(k+1-2m-\dim(F_0))+\epsilon}\right) \\
&= O_\epsilon(X^{-1+\epsilon}).
\end{align*} \\

If $m+1 \le k \le m+\dim(F_0)$, then: \\
\begin{align*}
J_m(U_2',X) &= X^{1-k} \int_{s_1, \ldots, s_n = c}^{CX} t_1^{-((n-2)-k+1)} t_m \cdots t_3 s_1 s_2^2 \cdots s_m^2 d^{\times}s \\
&= X^{1-k} \int_{s_1, \ldots, s_n = c}^{CX}  t_1^{-((n-2)-k+1)} t_3 \cdots t_{m} s_1 s_2^2 \cdots s_m^2 d^{\times}s \\
&= X^{1-k}  \int_{s_1, \ldots, s_n = c}^{CX}  t_1^{-((n-2)-k+1)} s_3 s_4^2 \cdots s_{m-1}^{m-3} s_1 s_2^2 \cdots s_m^2 d^{\times}s \\
&= O_\epsilon\left(X^{1-k+\frac{(m-3)(m-2)}{2}+2m-1+m(k+1-2m-\dim(F_0))+\epsilon}\right) \\
&= O_\epsilon(X^{-1+\epsilon}).
\end{align*} \\

If $m+\dim(F_0)+1 \le k \le n-2$, then: \\
\begin{align*}
J_m(U_2',X) &= X^{1-k} \int_{s_1, \ldots, s_n = c}^{CX} t_1^{-((n-2)-k+1)} t_{n-2-k+1} \cdots t_3 s_1 s_2^2 \cdots s_m^2 d^{\times}s \\
&= X^{1-k} \int_{s_1, \ldots, s_n = c}^{CX}  t_1^{-((n-2)-k+1)} t_3 \cdots t_{n-k-1} s_1 s_2^2 \cdots s_m^2 d^{\times}s \\
&= X^{1-k}  \int_{s_1, \ldots, s_n = c}^{CX}  t_1^{-((n-2)-k+1)} s_3 s_4^2 \cdots s_{n-k+1}^{n-k-1} s_1 s_2^2 \cdots s_m^2 d^{\times}s \\
&= O_\epsilon\left(X^{1-k+\frac{(n-k-1)(n-k)}{2}+2m-1+m(k+1-2m-\dim(F_0))+\epsilon}\right) \\
&= O_\epsilon(X^{-1+\epsilon}).
\end{align*} \\

Therefore, in all cases we find $$J_m(U_2',X) = O_\epsilon(X^{-1+\epsilon}).$$ The lemma now follows by induction on $m$ used to bound $I'_{m-1}(U_3',X)$ by $O_\epsilon(X^{-1+\epsilon})$. \\
\end{proof}

\begin{prop}
The number of absolutely irreducible elements in the cusp which have height at most $X$ is $O_\epsilon\left(X^{\left(\frac{n(n+1)}{2}-1\right)-1+\epsilon}\right)$. 
\end{prop}

We also find that the number of reducible elements in the main body is negligible. 

\begin{lem}
The number of integral points in the main body of $\mathcal{F}_A \cdot R_{A,b}^{r_2,\delta}$ which are not absolutely irreducible is bounded by $o\left(X^{\frac{n(n+1)}{2}-1}\right)$. 
\end{lem}
\begin{proof}
The proof is identical to the anisotropic case. 
\end{proof}

Therefore, we find the following asymptotic formula. 

\begin{thm}
Let $A \in \mathscr{L}_{\mathbb{Z}}$ be isotropic over $\mathbb{Q}$. We have $$N(V_{A,b}^{r_2,\delta}(\mathbb{Z});X) = \frac{1}{\sigma(r_2)} \Vol \left(\mathcal{F}_A \cdot R_{A,b}^{r_2,\delta}(X)\right) + o\left(X^{\frac{n(n+1)}{2}-1}\right).$$
\end{thm}

\begin{rem}
We can upgrade the results of this section to deal with $A$ not having the property that the $\mathbb{Q}$ degree of $\SO_A$ is equal to its $\mathbb{R}$ degree. To do so, it suffices to push through the arguments with generalised Siegel sets instead of the usual Siegel sets. We can deal with the compact parts of $M^0$ which appear by using the $o$-minimal version of Davenport's lemma (see \cite{ominimalDavenport}). 
\end{rem}

\section{Sieving to very large and acceptable collections}\label{Sieving to very large and acceptable collections}

In this section, we determine the desired asymptotic formulas. The results and proof contained in this section are adaptations of those of \cite{HoShankarVarmaOdd} to the case at hand and are repeated from Part I \cite{SiadOddMonogenicAverages} for the reader's convenience. We begin with the definition of a family of local specifications. 

\begin{defn}[Collection of local specifications and the associated set]
We say that a family $\Lambda_{A,b} = \left(\Lambda_{A,b,\nu} \right)_\nu$ of subsets $\Lambda_{A,b,\nu} \subset V_{A,b}(\mathbb{\mathcal{O}_\nu})$ indexed by the places $\nu$ of $\mathbb{Q}$ is a {\bf collection of local specifications} if: 1) for each finite prime $p$ the set $\Lambda_{A,b,p} \subset V_{A,b}(\mathbb{Z}_p) \setminus \left\{ \Delta = 0 \right\}$ is an open subset which is non-empty and whose boundary has measure $0$; and 2) at $\nu = \infty$, we have $\Lambda_{A,b,\infty} = V_{A,b}^{r_2,\delta}(\mathbb{R})$ for some integer $r_2$ with $0 \le r_2 \le \frac{n-1}{2}$ and $\delta \in \mathcal{T}(r_2)$. We associate the set $\mathcal{V}(\Lambda_{A,b}) := \left\{v \in V_{A,b}(\mathbb{Z}) \colon \forall \nu \left( v \in \Lambda_{A,b,\nu} \right) \right\}$ to  the collection of local specifications $\Lambda_{A,b} = \left(\Lambda_{A,b,\nu} \right)_\nu$. 
\end{defn}

\subsection{Sieving to projective elements} 

\begin{defn}
For a prime $p$, we denote by $V_{A,b}(\mathbb{Z}_p)^{\rm proj}$ the set of elements $v \in V_{A,b}(\mathbb{Z}_p)$ which correspond to a projective pair $(I,\delta)$ (i.e. with the property that $I^2 = (\delta)$) under the parametrisation. 
\end{defn}

We have $$ V_{A,b}^{r_2,{\rm proj}} (\mathbb{Z}) = V_{A,b}^{r_2}(\mathbb{Z}) \bigcap \left( \bigcap_p V_{A,b}^{\rm proj}(\mathbb{Z}) \right) .$$

\begin{defn}
We denote by $W_{A,b,p}$ the set of elements in $V_{A,b}(\mathbb{Z})$ that do not belong to $V_{A,b}^{\rm proj}(\mathbb{Z}_p)$. 
\end{defn}

As in Part I \cite{SiadOddMonogenicAverages}, we have the following estimate on the number of element of $W_{A,b,p}$ for large $p$. 

\begin{thm}
We have
\begin{equation*} 
N \left(\cup_{p \ge M} W_{A,b,p}, X \right) = O \left(\frac{X^{\frac{n(n+1)}{2}-1}}{M^{1-\epsilon}} \right)+o\left(X^{\frac{n(n+1)}{2}} \right)
\end{equation*} 
where the implied constant is independent of $X$ and $M$.
\end{thm}

We now define the concept of \emph{very large collections of local specifications} and state the asymptotic formula. Roughly, a collection of local specifications is very large if for large enough primes $p$, it includes all elements of $V$ which are projective at $p$. 

\begin{defn}[Very large collection of local specifications] Let $\Lambda_{A,b}= \left( \Lambda_{A,b,\nu}\right)_\nu$ be a collection of local specifications. We say that $\Lambda_{A,b}$ is {\bf very large} if for all but finitely many primes, the sets $\Lambda_{A,b,p}$ contains all projective elements of $V_{A,b}(\mathbb{Z}_p)$. If $\Lambda_{A,b}$ is very large, we also say that the associated set $\mathcal{V}(\Lambda_{A,b})$ is very large. 
\end{defn}

\begin{thm}
Let $r_2$ be an integer such that $0 \le r_2 \le \frac{n-1}{2}$ and let $\delta \in \mathcal{T}(r_2)$. Then for a very large collection of local specifications $\Lambda_{A,b}$ such that $\Lambda_{A,b,\infty} = V_{A,b}^{r_2,\delta}(\mathbb{R})$, we have 
\begin{equation*} 
N(\mathcal{V}(\Lambda_{A,b}^\delta), X) = \frac{1}{\sigma(r_2)} \Vol(\mathcal{F}_A \cdot R_{A,b}^{r_2,\delta}(X)) \prod \limits_p \Vol(\Lambda_{A,b,p}) + o\left(X^{\frac{n(n-1)}{2}-1}\right),
\end{equation*}
where the volume of subsets of $V_{A,b}(\mathbb{R})$ are computed with respect to the Euclidean measure normalized so that $V_{A,b}(\mathbb{Z})$ has covolume 1 and the volumes of subsets of $V_{A,b}(\mathbb{Z}_p)$ are computed with respect to the Euclidean measure normalized so that $V_{A,b}(\mathbb{Z}_p)$ has measure $1$. 
\end{thm}

\subsection{Sieving to acceptable sets conditional on a tail estimate} 

We now define the concept of \emph{acceptable collections of local specifications} and state the asymptotic formula. Roughly, a collection of local specifications is acceptable if for large enough primes $p$, it includes all fields with discriminant indivisible by $p^2$. 

\begin{defn}[Acceptable collection of local specifications]
Let $\Lambda_{A,b}= \left( \Lambda_{A,b,\nu}\right)_\nu$ be a collection of local specifications. We say that $\Lambda_{A,b}$ is {\bf acceptable} if for all but finitely many primes, the set $\Lambda_{A,b,p}$ contains all elements of $V_{A,b}(\mathbb{Z}_p)$ whose discriminant is not divisible by $p^2$. If $\Lambda_{A,b}$ is acceptable, we also say that the associated set $\mathcal{V}(\Lambda_{A,b})$ is acceptable. 
\end{defn}

We have the following unconditional asymptotic inequality. 

\begin{thm}
Let $\Lambda_{A,b}= \left( \Lambda_{A,b,\nu}\right)_\nu$ be an acceptable collection of local specifications. 
\begin{equation*}
N(\mathcal{V}(\Lambda_{A,b}),X) \le \frac{1}{\sigma(r_2)}\Vol(\mathcal{F}_A \cdot R_{A,b}^{r_2,\delta}(X)) \prod \limits_p \Vol(\Lambda_{A,b,p}) + o\left(X^{\frac{n(n+1)}{2}-1} \right),
\end{equation*}
where the volume of subsets of $V_{A,b}(\mathbb{R})$ are computed with respect to the Euclidean measure normalized so that $V_{A,b}(\mathbb{Z})$ has covolume $1$ and the volumes of subsets of $V_{A,b}(\mathbb{Z}_p)$ are computed with respect to the Euclidean measure normalized so that $V_{A,b}(\mathbb{Z}_p)$ has measure $1$. 
\end{thm}

The following tail estimates are known for $n=4$ as will be shown in a forthcoming work with Arul Shankar and likely to be true for $n \ge 6$. Indeed they follow from a suitable version of the $abc$ conjecture by work of Granville.

\begin{defn}
Let $p$ be a prime. We denote by $\mathcal{W}_{A,b,p}$ the set of elements $v \in V_{A,b}(\mathbb{Z})$ such that $p^2 \mid \Delta(v)$. 
\end{defn}

\begin{conj}[Conjectural tail estimates] We have
\begin{equation*}
N(\cup_{p \ge M} \mathcal{W}_{A,b,p}, X) = O \left(\frac{X^{\frac{n(n+1)}{2}-1}}{M^{1-\epsilon}}\right) + o \left( X^{\frac{n(n+1)}{2}-1} \right)
\end{equation*}
where the implied constant is independent of $X$ and $M$.
\end{conj}

We have the following asymptotic formula conditional on the preceding tail estimates. 

\begin{thm}
Suppose that the preceding tail estimates hold. Let $r_2$ be an integer such that $0 \le r_2 \le \frac{n-1}{2}$ and let $\delta \in \mathcal{T}(r_2)$. Then for an acceptable collection of local specifications $\Lambda_{A,b}$ such that $\Lambda_{A,b}(\infty) = V_{A,b}^{r_2,\delta}(\mathbb{R})$ we have 
\begin{equation*}
N(\mathcal{V}(\Lambda_{A,b}^\delta), X) = \frac{1}{\sigma(r_2)} \Vol(\mathcal{F}_A \cdot R_A^{r_2,\delta}(X)) \prod \limits_p \Vol(\Lambda_{A,b,p}) + o\left(X^{\frac{n(n-1)}{2}-1}\right),
\end{equation*}
where the volume of subsets of $V_{A,b}(\mathbb{R})$ are computed with respect to the Euclidean measure normalized so that $V_{A,b}(\mathbb{Z})$ has covolume $1$ and the volumes of subsets of $V_{A,b}(\mathbb{Z}_p)$ are computed with respect to the Euclidean measure normalized so that $V_{A,b}(\mathbb{Z}_p)$ has measure $1$. 
\end{thm}

\section{Change of measure formula} \label{Change of measure formula}

To compute the volumes of sets and multi-sets in $V_{A,b}(\mathbb{R})$ and $V_{A,b}(\mathbb{Z}_p)$, we have the following version of the change of variable formula. Let $dv$ and $df$ denote the Euclidean measure on $V_{A,b}$ and $U_{A,b}$ respectively normalized so that $V_{A,b}(\mathbb{Z})$ and $U_{A,b}(\mathbb{Z})$ have co-volume $1$. Furthermore, let $\omega$ be an algebraic differential form generating the rank $1$ module of top degree left-invariant differential forms on $\SO_A$ over $\mathbb{Z}$.

\begin{prop}[Change of measure formula]
Let $K =\mathbb{Z}_p, \mathbb{R}$ or $ \mathbb{C},$. Let $| \cdot|$ denote the usual absolute value on $K$ and let $s \colon U_{1,b}(K) \rightarrow V_{A,b}(K)$ be a continuous map such that $\pi(f) = $ for each $f \in U_{1,b}$. Then there exists a rational non-zero constant $\mathcal{J}_A$, independent of $K$ and $s$, such that for any measurable function $\phi$ on $V_{A,b}(K)$, we have: 
\begin{equation*} 
\int \limits_{\SO_A(K) \cdot s(U_{1,b}(K))} \phi(v) \, dv = \left| \mathcal{J}_A \right| \int_{f\in U_{1,b}(K)} \int_{g \in \SO_A(K)} \phi(g \cdot s(f)) \, \omega(g) \, df
\end{equation*}
\begin{equation*} 
\int \limits_{V_{A,b}(K)} \phi(v) dv = \left| \mathcal{J}_A \right| \int\limits_{\substack{f \in U_{1,b}(K) \\ \Delta(f) \neq 0}} \left(\sum_{v\in \frac{V_{A,b}(K) \cap \pi^{-1}(f)}{\SO_A(K)}} \frac{1}{\# {\rm Stab}_{\SO_A(\mathbb{Z}_p)}(v)} \int\limits_{g \in \SO_A(K)} \phi(g \cdot v) \, \omega(g) \right) \, df
\end{equation*}
where $\frac{V_{A,b}(K) \cap \pi^{-1}(f)}{\SO_A(K)}$ denotes a set of representatives for the action of $\SO_A(\mathbb{Z}_p)$ on $V_{A,b}(\mathbb{Z}_p) \cap \pi^{-1}(f)$. 
\end{prop}

We can simplify the second integral above by introducing a local mass. 

\begin{defn}[Local mass formula] Let $p$ be a prime, $f \in U_{1,b}(\mathbb{Z}_p)$ and $A \in \mathscr{L}_{\mathbb{Z}}$. We define the local mass of $f$ at $p$ in $A$, $m_p(f,A)$ to be $$m_p(f,A) := \sum_{v \in \frac{V_{A,b}(\mathbb{Z}_p) \cap \pi^{-1}(f)}{\SO_A(\mathbb{Z}_p)}} \frac{1}{\# {\rm Stab}_{\SO_A(\mathbb{Z}_p)}(v)}.$$
\end{defn}

We now have the following formula for the local volumes appearing in the asymptotic formula. 

\begin{prop} We have $$\Vol\left( \mathcal{F}_A \cdot R_{A,b}^{r_2,\delta} (X) \right) =  \chi_A(\delta) \left| \mathcal{J}_A \right| \Vol( \mathcal{F}_A^\delta) \Vol(U(\mathbb{R})^{r_2}_{H < X}).$$ Let $S_p \subset U_{1,b}(\mathbb{Z}_p)$ be a non-empty open set whose boundary has measure $0$. Consider the set $\Lambda_{A,b,p} = V_{A,b}(\mathbb{Z}_p) \cap \pi^{-1}(S_p)$. Then we have $$\Vol( \Lambda_{A,b,p}) = \left| \mathcal{J}_A \right|_p \Vol(\SO_A(\mathbb{Z}_p)) \int_{f \in S_p} m_p(f,A) \, df.$$
\end{prop}

\section{The product of local volumes and the local mass} \label{The product of local volumes and the total local mass}

The calculation of the total mass is slightly more delicate here, and we divide it into a couple of lemmas. 

\begin{lem}[Total quantity] Let $R$ be a non-degenerate ring of degree $n$ over $\mathbb{Z}_p$. The quantity
\begin{equation*}
\frac{\left| R^\times/(R^\times)^2 \right|}{\left|R^\times[2] \right|}
\end{equation*}
is equal to $1$ if $p \neq 2$ and to $2^n$ if $p = 2$. 
\end{lem}

\begin{proof}
The following sequence is exact: 
\begin{equation*}
0 \longrightarrow R^\times[2] \longrightarrow R^\times \overset{ (\cdot)^2}\longrightarrow R^\times \longrightarrow R^\times/(R^\times)^2 \longrightarrow 0.
\end{equation*}
Let us note that $R^\times$ is the direct product of a finite abelian subgroup $F$ and $\mathbb{Z}_p^n$ (written additively). The exact sequence above is thus the direct sum of the corresponding exact sequence on $F$ and on $\mathbb{Z}_p^n$. Computing the Euler characteristic on $F$ we find: 
$$\frac{\left| F/ F^2 \right|}{\left| F[2] \right|} = 1.$$
The exact sequence on the free part takes the form: $$0 \longrightarrow \mathbb{Z}_p^n \overset{\times 2}\longrightarrow \mathbb{Z}_p^n \longrightarrow \mathbb{Z}_p^n / (2 \mathbb{Z}_p)^n \longrightarrow  0.$$ In order to extract information from this part, we need to treat the cases $p \neq 2$ and $p = 2$ separately. For $p \neq 2$, the map $\times 2$ is surjective so the free part contributes a factor of $1$. For $p = 2$, the module $(2\mathbb{Z}_2)^n$ has index $2^n$ in $\mathbb{Z}_2^n$ and so the free part contributes a factor of $2^n$. This completes the proof the lemma. 
\end{proof}

\begin{lem}[Norm isolation]
Let $R$ be a non-degenerate ring of degree $n$ over $\mathbb{Z}_p$. Let $N$ denote the norm map from $N \colon R^\times/(R^\times)^2 \rightarrow \mathbb{Z}_p^\times/ (\mathbb{Z}_p^\times)^2$. The quantity
\begin{equation*}
\frac{\left| (R^\times/(R^\times)^2)_{N \equiv 1} \right|}{\left|R^\times[2] \right|} 
\end{equation*}
is equal to $\frac{1}{\left| N(R^\times) \right|}$ if $p \neq 2$ and to $\frac{2^n}{\left| N(R^\times) \right|}$ if $p = 2$. 
\end{lem}
\begin{proof}
The first isomorphism theorem gives us $N(R^\times) \cong \frac{ (R^\times/(R^\times)^2)}{ (R^\times/(R^\times)^2)_{N \equiv 1} }$. Since everything is finite we get: $$ \left| R^\times/(R^\times)^2)_{N \equiv 1} \right| = \frac{\left| R^\times/(R^\times)^2) \right|}{\left| N(R^\times) \right|}$$ and the result follows from the previous lemma. 
\end{proof}

%\begin{proof}
%The following sequence is exact: 
%\begin{equation*}
%0 \longrightarrow R^\times[2] \longrightarrow R^\times \overset{ (\cdot)^2}\longrightarrow R^\times_{N \equiv 1} \longrightarrow R^\times_{N \equiv 1}/(R^\times)^2 \longrightarrow 0.
%\end{equation*}
%Let us note that $R^\times$ is the direct product of a finite abelian subgroup $F$ and $\mathbb{Z}_p^n$. 
%
%Now, if $p \neq 2$ then every element of $\mathbb{Z}_p^n$ is in the image of the multiplication by $2$ map. It thus suffices to examine the sequence above on the finite abelian part of $R^\times$. There we have an exact sequence of finite abelian groups. It's Euler characteristic is thus equal to $1$ and so we have $$1 = \frac{\left|F \right|}{\left| F_{N \equiv 1}  \right|} \frac{\left| (F/ F^2)_{N \equiv 1} \right|}{\left| F[2] \right|}.$$ Now, $\frac{F}{F_{N \equiv 1}}$ is isomorphic to the image of the norm map from $R^\times$ to $\mathbb{Z}_p^\times / (\mathbb{Z}_p^\times)^2 \cong \mathbb{Z}/2\mathbb{Z}$. Therefore, the quantity is equal to $\frac{1}{2}$ if norm map to $\mathbb{Z}_p^\times / (\mathbb{Z}_p^\times)^2$ is surjective and to $1$ otherwise. 
%
%We now examine the case $p = 2$. In this case, the multiplication by $2$ map on $\mathbb{Z}_2^n$ is not surjective. Just as in \cite{HoShankarVarmaOdd}, we pick up a factor of $2^n$ from looking at the index of $\mathbb{Z}_2$-module $2 \mathbb{Z}^n_2$ in $\mathbb{Z}^n_2$. The rest of the argument goes through and we find that the quantity \ref{star} is equal to $1$ over the size of $\frac{F}{F_{N \equiv 1}}$.
%\end{proof}

In particular, we get the following statement concerning the total local masses for maximal rings which are not evenly ramified at $2$.  It follows from the previous lemma and the fact that if $\mathcal{O}^\times$ is the ring of integers of a finite extension of $\mathbb{Q}_p$, the norm map $N \colon \mathcal{O}^\times/(\mathcal{O}^\times)^2 \rightarrow \mathbb{Z}_p^\times/(\mathbb{Z}_p^\times)^2$ is surjective when the extension has odd ramification degree and not surjective when $p \neq 2$ and the extension has even ramification degree.

\begin{lem}
The total local mass for $f \in \mathbb{Z}_p[x]$ maximal and not evenly ramified at $p$ is given by: 
\begin{equation*}
m_p(f) = \begin{cases}
2^{n-1} & \text{ if } p = 2 \\
1 & \text{ if } p \neq 2 
\end{cases}.
\end{equation*}
The total local mass for $f \in \mathbb{Z}_p[x]$, $p \neq 2$, maximal and evenly ramified is given by: 
\begin{equation*}
m_p(f) = 2.
\end{equation*}
\end{lem}

\begin{rem}
Thus, even ramification has a doubling effect on the total local mass. Heuristically, this is because if $f$ has even ramification at the prime $p$, then the ideal $(p)$ splits into a product of prime ideals $(p) = \mathfrak{P}_1^{e_1}\cdots \mathfrak{P}_m^{e_m}$ each appearing to an even power $e_i$, and thus we can write $(p) = I^2$. So we have constructed a $2$-torsion ideal in the oriented ideal class group. The fact that the total mass is twice as large means that on average, these ideals contribute one extra generator to the $2$-torsion in the oriented ideal class group.
\end{rem}

\subsection{Computing the local masses} We now define the infinite mass and compute $m_p(f,A)$ for all $p \neq 2,\infty$. 

\begin{defn}[The infinite mass]
Let $A \in \mathscr{L}_\mathbb{Z}$ and $r_2$ be an integer such that $0 \le r_2 \le \frac{n-1}{2}$. The infinite mass of $A$ with respect to $r_2$ is defined to be $$m_\infty(r_2,A) = \sum_{\delta \in \mathcal{T}(r_2)} \chi_A(\delta).$$
\end{defn}

The following lemma isolates the main properties of the local masses for all $p$, including the Archimedean place. 

\begin{lem}[Main properties of the local masses] The local masses $m_p(f,A)$ and $m_\infty(A)$ have the following properties.
\begin{enumerate}
\item If $\gamma \in \SL_n(\mathbb{Z}_p)$, we have $$ m_p(f,\gamma^t A \gamma) = m_p(f,A).$$
\item If $\gamma \in \SL_n(\mathbb{R})$, we have $$m_\infty(r_2,\gamma^t A \gamma) = m_\infty(r_2,A).$$
\item In particular, if $A_1$ and $A_2$ are unimodular integral matrices lying in the same genus, we have $$m_p(f,A_1) = m_p(f,A_2)$$ for all primes $p$ and $$m_\infty(r_2,A_1) = m_\infty(r_2,A_2).$$ 
\item The sum of $m_2(f,A)$ over a set of representatives for the unimodular orbits of the action of $\SL_n(\mathbb{Z}_2)$ on $\Sym_n(\mathbb{Z}_2)$ is $$\sum_{\substack{A \in \frac{\Sym_n(\mathbb{Z}_2)}{\SL_n(\mathbb{Z}_2)} \\ \det(A) = 1 \in \frac{\mathbb{Z}_2}{\mathbb{Z}_2^2}}} m_2(f,A) = m_2(f).$$
\item The sum of $m_p(f,A)$ over a set of representatives for the unimodular orbits of the action of $\SL_n(\mathbb{Z}_p)$ on $\Sym_n(\mathbb{Z}_p)$ is $$\sum_{\substack{A \in \frac{\Sym_n(\mathbb{Z}_p)}{\SL_n(\mathbb{Z}_p)} \\ \det(A) = 1 \in \frac{\mathbb{Z}_p}{\mathbb{Z}_p^2}}} m_p(f,A) = m_p(f).$$
\item The sum of $m_\infty(r_2,A)$ over a set of representatives for unimodular the orbits of the action of $\SL_n(\mathbb{R})$ on $\Sym_n(\mathbb{R})$ is $$\sum_{\substack{A \in \frac{\Sym_n(\mathbb{R})}{\SL_n(\mathbb{R})} \\ \det(A) = 1 \in \frac{\mathbb{R}}{\mathbb{R}^2}}} m_\infty(r_2,A) = 2^{r_1-1}$$
if $r_1 > 0$ and 
$$\sum_{\substack{A \in \frac{\Sym_n(\mathbb{R})}{\SL_n(\mathbb{R})} \\ \det(A) = 1 \in \frac{\mathbb{R}}{\mathbb{R}^2}}} m_\infty(r_2,A) = 2$$
if $r_1 = 0$.
\end{enumerate}

\end{lem}

We can now compute the local masses for all $p \neq 2,\infty$ by noting that there is a unique $\SL_n(\mathbb{Z}_p)$ equivalence class of bilinear forms with determinant $1$. 

\begin{cor}[Local masses for $p \neq 2,\infty$] For $A \in \mathscr{L}_\mathbb{Z}$  and $p \neq 2,\infty$ we have \begin{equation*}
m_p(f,A) = \begin{cases}
2 & \text{ if }f\text{ is evenly ramified at }p \\
1 & \text{ otherwise }
\end{cases}.
\end{equation*}
\end{cor}

The computation of the local masses at $p = 2, \infty$ is more delicate and is the object of the following sections.

\section{Point count and the $2$-adic mass} \label{Point count and the $2$-adic mass}

We compute the $2$-adic mass for polynomials which are not evenly ramified at the prime $2$. We begin by showing that this restriction allows us to exclude type ${\rm II}$ genera from our considerations. 

\begin{thm}[Type II genera and over-ramification at $2$]
Let $A$ be the hyperbolic unimodular symmetric bilinear form over $\mathbb{F}_2$. Then the image under the resolvent map of $V_A(\mathbb{F}_2)$ lies in $$\mathbb{F}_2[x^2] = \Big(\mathbb{F}_2[x]\Big)^2.$$ 

Therefore, pairs whose first component is hyperbolic modulo $2$ correspond to ideal classes of monogenic orders which are ``over-ramified'' at $2$ in the sense that $(2)$ splits as a product of even powers of primes ideals. In particular, the discriminant of the resolvent polynomials of such pairs is not squarefree at $2$. Furthermore, the statement of the theorem holds if $\mathbb{F}_2$ is replaced by any ring of characteristic $2$. 
\end{thm}
\begin{proof}
Without loss of generality we can let $A$ be the matrix with $1\,$s on the anti-diagonal and $B$ be any symmetric matrix. We proceed by examining the resolvent polynomial from the perspective of universal algebra. Let us introduce the multivariate polynomial $$\Phi(X_1,\ldots,X_{\frac{n}{2}},Y_{\frac{n}{2}},\ldots,Y_1) \in \mathbb{F}_2[X_1,\ldots,X_{\frac{n}{2}},Y_{\frac{n}{2}},\ldots,Y_1]$$ which we define as the determinant: 

$$\left|\begin{array}{ccccc} 
b_{1 \, 1} & b_{1 \, 2} & \vdots & b_{1 \, n-1} & Y_1 + b_{1 \, n}  \\ 
b_{2 \, 1} & b_{2 \, 2} & \vdots & Y_{2} + b_{2 \, n-1}  & b_{2 \, n}  \\ 
\cdots & \cdots & \iddots & \cdots & \cdots  \\ 
b_{n-1 \, 1} & X_2+b_{n-1 \, 2} & \vdots & b_{n-1 \, n-1} & b_{n-1 \, n}  \\ 
X_1+b_{n \, 1} & b_{n \, 2} & \vdots & b_{n \, n-1} & b_{n \, n} \\
\end{array} \right|.$$

For this polynomial, we claim that there is a natural bijection between the set of monomials which are divisible by $Y_i$ but not $X_i$ and the set of monomials which are divisible by $X_i$ but not $Y_i$. Indeed, the map which exchanges $X_i$ and $Y_i$ is easily seen to give a bijection since the matrix $(b_{ij})$ is symmetric and $-1 = 1$ in a ring of characteristic $2$. 

We can use this observation to deduce that $\Phi(X,\ldots,X,X,\ldots,X) = \pi(A,B)$ contains only monomials of even degree. For this purpose, it is suffices to show that the only monomials appearing in $\Phi(X_1,\ldots,X_{\frac{n}{2}},X_{\frac{n}{2}},\ldots,X_1)$ are those of even degree (the degree of the monomial $X_1^{e_1}X_2^{e_2}\cdots X_{\frac{n}{2}}^{e_n}$ is defined to be $e_1 + \ldots + e_n$). But this follows immediately from the observation made in the previous paragraph. Indeed, a monomial of odd degree in the polynomial $\Phi(X_1,\ldots,X_{\frac{n}{2}},X_{\frac{n}{2}},\ldots,X_1)$ comes from a sum of monomials in $\Phi(X_1,\ldots,X_{\frac{n}{2}},Y_{\frac{n}{2}},\ldots,Y_1)$, the elements of this sum coming in pairs one of which is divisible by $X_i$ but not by $Y_i$ for some $1 \le i \le \frac{n}{2}$, the other of which is divisible by $Y_i$ but not by $X_i$, and which are such that exchanging the variable $X_i$ and $Y_i$ maps each of the monomials to the other. 
\end{proof}

We can now apply the methods of \cite{SiadOddMonogenicAverages} to obtain the $2$-adic masses for polynomials which are not over-ramified at the prime $2$.

By the classification of quadratic forms over $\mathbb{Z}_2$, there are only two odd, determinant $(-1)^{\frac{n}{2}}$ classically integral quadratic forms over $\mathbb{Z}_2$ of even dimension $n$ up to $\SL_n(\mathbb{Z}_2)$ equivalence. See for instance \cite{JonesCanonical} or \cite{ConwaySloaneClassification}. 

For $n \equiv 0 \mod 4$ we can take:  
$$\mathfrak{M}_1 = \begin{pmatrix}
1 &   &  &  & &  \\ 
 & \ddots &  &  & & \\ 
 &  & 1 &  &  &\\ 
 &  &  & 1 &  &\\ 
 &  &  &  & 1 &\\
 &  &  &  &  & 1\\
\end{pmatrix}, \,\,\,\,
\mathfrak{M}_{-1} = \begin{pmatrix}
1 &   &  &  & &  \\ 
 & \ddots &  &  & & \\ 
 &  & 1 &  &  &\\ 
 &  &  & 1 &  &\\ 
 &  &  &  & -1 &\\
 &  &  &  &  & -1\\
\end{pmatrix}.$$

For $n \equiv 2 \mod 4$ we can take:  
$$\mathfrak{M}_1 = \begin{pmatrix}
1 &   &  &  & &  \\ 
 & \ddots &  &  & & \\ 
 &  & 1 &  &  &\\ 
 &  &  & 1 &  &\\ 
 &  &  &  & 1 &\\
 &  &  &  &  & -1\\
\end{pmatrix}, \,\,\,\,
\mathfrak{M}_{-1} = \begin{pmatrix}
1 &   &  &  & &  \\ 
 & \ddots &  &  & & \\ 
 &  & 1 &  &  &\\ 
 &  &  & -1 &  &\\ 
 &  &  &  & -1 &\\
 &  &  &  &  & -1\\
\end{pmatrix}.$$

\begin{rem}
We have chosen the subscripts above to match the Hasse--Witt symbol of the respective bilinear forms.
\end{rem} 

We now state the constraint imposed on the total mass by the local mass. 

\begin{lem}
For any $f \in U_{1,b}(\mathbb{Z}_2)$, we have $$m_2(f,\mathfrak{M}_1)+m_2(f,\mathfrak{M}_{-1}) = 2^{n-1}.$$
\end{lem}

We give names to the integral of the $2$-adic masses over $S_2$.

\begin{defn}
Define
\begin{align*}
c_2(n,\mathfrak{M}_1) &= \int_{f \in S_2} m_2(f,\mathfrak{M}_1) \, df \\
c_2(n,\mathfrak{M}_{-1}) &= \int_{f \in S_2} m_2(f,\mathfrak{M}_{-1}) \, df.
\end{align*}
\end{defn}

\begin{lem} [Point count]
Let $S_2 \subset U_{1,b}(\mathbb{Z}_2)$ be a local condition on the space of non-over-ramified monic polynomials at the prime $2$ defined modulo $2$. Denote by $\Lambda_{\mathfrak{M}_1}(2)$ and $\Lambda_{\mathfrak{M}_{-1}}(2)$ the pre-images in $V_{\mathfrak{M}_1,b}(\mathbb{Z}_2)$ and $V_{\mathfrak{M}_{-1},b}(\mathbb{Z}_2)$ respectively of $S_2$ under the resolvent map $\pi$. Then the volumes of these two sets are equal $$\Vol(\Lambda_{\mathfrak{M}_1}(2)) = \Vol(\Lambda_{\mathfrak{M}_{-1}}(2)).$$
\end{lem}
\begin{proof}
The canonical representatives $\mathfrak{M}_1$ and $\mathfrak{M}_{-1}$ are equal modulo $2$. Since $\Lambda_{\mathfrak{M}_1}(2)$ and $\Lambda_{\mathfrak{M}_{-1}}(2)$ are defined by imposing congruence conditions modulo $2$ on $V_{\mathfrak{M}_1,b}(\mathbb{Z}_2)$ and $V_{\mathfrak{M}_{-1},b}(\mathbb{Z}_2)$, the result follows. 
\end{proof}

As in \cite{SiadOddMonogenicAverages} we find. 

\begin{lem} 
The rational numbers giving the Jacobian change of variables for $\mathfrak{M}_{1}$ and $\mathfrak{M}_{-1}$ are equal when the volume forms on the associated special orthogonal groups are those associated to point counting modulo increasing powers of $p$: $$\mathcal{J}_{\mathfrak{M}_{1}} =   \mathcal{J}_{\mathfrak{M}_{-1}}.$$ In particular, their $2$-adic valuations are the same $$\left| \mathcal{J}_{\mathfrak{M}_{1}} \right|_2 = \left| \mathcal{J}_{\mathfrak{M}_{-1}} \right|_2.$$ 
\end{lem}

Finally, we compute the ratio of the $2$-adic masses by finding the ratio between the volumes of the $2$-adic points of the special orthogonal groups $\SO_{\mathfrak{M}_1}$ and $\SO_{\mathfrak{M}_{-1}}$. 

\begin{prop}[Volume ratio for Type ${\rm I}$ genera] We have
\begin{align*}
\frac{c_2\left(n,\mathfrak{M}_1\right)}{c_2 \left(n,\mathfrak{M}_{-1}\right)} = \frac{\Vol(\Lambda_{\mathfrak{M}_1}(2)) \Bigg(\left| \mathcal{J}_{\mathfrak{M}_{-1}} \right|_2 \Vol(\SO_{\mathfrak{M}_{-1}}(\mathbb{Z}_2)) \Bigg)}{\Vol(\Lambda_{\mathfrak{M}_{-1}}(2)) \Bigg(\left| \mathcal{J}_{\mathfrak{M}_1} \right|_2 \Vol(\SO_{\mathfrak{M}_1}(\mathbb{Z}_2)) \Bigg)} &= \frac{\Vol(\SO_{\mathfrak{M}_{-1}}(\mathbb{Z}_2))}{\Vol(\SO_{\mathfrak{M}_1}(\mathbb{Z}_2))} \\
&=\frac{2^{n-2} \pm_8 2^{\frac{n-2}{2}}}{2^{n-2} \mp_8 2^{\frac{n-2}{2}}},
\end{align*}
where $\pm_8$ is $+$ if $n$ is congruent to $0$ or $2 \mod 8$ and $-$ otherwise.
\end{prop}
\begin{proof}
This calculation is the same as in \cite{SiadOddMonogenicAverages}. Care is needed to  keep track of the of the \emph{octane values} of $\mathfrak{M}_1$ and $\mathfrak{M}_{-1}$ for different congruence classes of $n$ modulo $8$ resulting in the cases for $\pm_8$. We find that $\mathfrak{M}_1$ has octane value $0 \Mod 8$ if $n \equiv 0 \Mod 8$, $4 \Mod 8$ if $n \equiv 4 \Mod 8$, $0 \Mod 8$ if $n \equiv 2 \Mod 8$, and $4 \Mod 8$ if $n \equiv 6 \Mod 8$. On the other hand, we find that $\mathfrak{M}_{-1}$ has octane value $4 \Mod 8$ if $n \equiv 0 \Mod 8$, $0 \Mod 8$ if $n \equiv 4 \Mod 8$, $4 \Mod 8$ if $n \equiv 2 \Mod 8$, and $0 \Mod 8$ if $n \equiv 6 \Mod 8$.
\end{proof}

We thus obtain the values for the 2-adic mass. 

\begin{cor}
The $2$-adic masses satisfy the following identities: 
\begin{align*}
c_2(n,\mathfrak{M}_1) + c_2(n,\mathfrak{M}_{-1}) &= 2^{n-1}\Vol(S_2) \\
c_2(n,\mathfrak{M}_1) - c_2(n,\mathfrak{M}_{-1}) &= \pm_8 2^{\frac{n-2}{2}+1} \Vol(S_2).
\end{align*}
In particular: 
\begin{alignat*}{2}
&c_2(n,\mathfrak{M}_1) &&= \left(2^{n-2} \pm_8 2^{\frac{n-2}{2}} \right) \Vol(S_2)\\
&c_2(n,\mathfrak{M}_{-1}) &&= \left(2^{n-2} \mp_8 2^{\frac{n-2}{2}} \right) \Vol(S_2)
\end{alignat*}
where $\pm_8$ is $+$ if $n$ is congruent to $0$ or $2 \mod 8$ and $-$ otherwise.
\end{cor}

\section{The infinite mass} \label{The infinite mass}

In this section, we calculate the infinite masses. We begin by describing the distribution of $\delta$ among the different $A$ slices as in \cite{SiadOddMonogenicAverages}. 

\begin{defn}
Let $\delta \in \mathcal{T}(r_2)$. Define $\Omega_-(\delta)$ to be the number of negative eigenvalues of the first component of the orbit associated to $\delta$. 
\end{defn}

\begin{thm}[Signature distribution of first components of $\mathcal{T}(r_2)$] If $r_1 > 0$, the number of $\delta$ in $\mathcal{T}(r_2)$ such that $\Omega_-(\delta) = q$ is $${r_1 \choose q-r_2}.$$ In particular, if $q < r_2$, there are no $\delta \in \mathcal{T}(r_2)$ which land in any $V_A$ for which $A$ has signature $(n-q,q)$. If $r_1 = 0$, there are $2$ orbits which both land on the split slice. 
\end{thm}
\begin{proof}
The proof is the same as in \cite{SiadOddMonogenicAverages}. 
\end{proof}

To make the final calculation more transparent, we state the relation between the value of $\Omega_-(\cdot)$ and the Hasse--Witt symbol. We also obtain a couple of useful combinatorial identities. \\

If $n \equiv 0 \Mod 4$, we are considering real bilinear forms with determinant $1$. Among those, the ones which satisfy $\Omega(\delta) \equiv 0 \Mod 4$ have Hasse--Witt symbol $1$, while those which satisy $\Omega(\delta) \equiv 2 \Mod 4$ have Hasse--Witt symbol $-1$. \\

If $n \equiv 2 \Mod 4$, we are considering real bilinear forms with determinant $-1$. Among those, the ones which satisfy $\Omega(\delta) \equiv 1 \Mod 4$ have Hasse--Witt symbol $1$, while those which satisy $\Omega(\delta) \equiv 3 \Mod 4$ have Hasse--Witt symbol $-1$. \\

We now define the infinite mass of an integral quadratic form $A \in \mathscr{L}_{\mathbb{Z}}$ to be the number of $\delta \in \mathcal{T}(r_2)$ whose associated orbit in $V$ intersects $V_A$ non-trivially. 

\begin{defn}[The infinite mass]
Let $A \in \mathscr{L}_\mathbb{Z}$ and $r_2$ be an integer such that $0 \le r_2 \le \frac{n-1}{2}$. The infinite mass of $A$ with respect to $r_2$ is defined to be $$m_\infty(r_2,A) = \sum_{\delta \in \mathcal{T}(r_2)} \chi_A(\delta).$$
\end{defn}

Above, we found that these infinite masses were equal to specific binomial coefficients. We now compute closed forms for certain \emph{total infinite masses} arising when summing infinite masses over genera having the same $2$-adic reduction. 

\begin{defn}[The total infinite mass] Let $\mathcal{G}_2$ be set of equivalence classes of unimodular bilinear forms over $\mathbb{Z}_2$ of dimension $n$. Then $\mathcal{G}_2 = \{\mathfrak{M}_{-1}, \mathfrak{M}_1, \mathfrak{M}_{\textrm{Type II}}\}$ where $\mathfrak{M}_{-1}$ and $\mathfrak{M}_1$ are the odd unimodular forms with Hasse--Witt symbol $-1$ and $1$ respectively (canonical representative were described in the last section) and $\mathfrak{M}_{\textrm{Type II}}$ is simply the bilinear form with $1$ on the anti-diagonal. Note that $\mathfrak{M}_{\textrm{Type II}}$ has Hasse--Witt symbol $1$ if $n \equiv 0,2 \Mod 8$ and $-1$ if $n \equiv 4,6 \Mod 8$. The \emph{total infinite mass}, $c_\infty(r_1, \mathfrak{M}_{\blacksquare} )$, associated to an element $\mathfrak{M}_{\blacksquare} \in \mathcal{G}_2$ is defined to be the sum of the infinite masses over all genera whose $\mathbb{Z}_2$ reduction coincides with $\mathfrak{M}_{\blacksquare}$. 
\end{defn}

Calculating the total infinite mass now becomes a question of evaluating sums of binomial coefficients in arithmetic progression.

\begin{cor}
The infinite masses are as follows:
\begin{align*}
c_\infty(r_1,\mathfrak{M}_{1}) &= 2^{r_1-2}\pm_8 2^{\frac{r_1-2}{2}} \\
c_\infty(r_1,\mathfrak{M}_{-1}) &= 2^{r_1-2}\mp_8 2^{\frac{r_1-2}{2}}.
\end{align*}
where $\pm_8$ is $+$ if $n$ is congruent to $0$ or $2 \Mod 8$ and $-$ otherwise. Furthermore, $c_\infty(r_1,\mathfrak{M}_{\textrm{Type II}}) = c_\infty(r_1,\mathfrak{M}_{1})$ if $n \equiv 0,2 \Mod 8$ and $c_\infty(r_1,\mathfrak{M}_{\textrm{Type II}}) = c_\infty(r_1,\mathfrak{M}_{-1})$ if $n \equiv 4,6 \Mod 8$.
\end{cor}

\section{Statistical consequences} \label{Statistical consequences}

We now pool together the elements assembled in the previous sections to compute the averages. We treat the cases $r_1 = 0$ and $r_1 >0$ separately starting with totally imaginary orders. For simplicity, we present the computation for fields which are non-evenly ramified at all primes $p$.  We state the general theorem in the last subsection.

\subsection{Oriented class group averages for non-evenly ramified totally imaginary fields} 

\begin{align*}
&\frac{\sum\limits_{\substack{\mathcal{O} \in \mathfrak{R}, \\ H(\mathcal{O}) < X}} 2^{r_2} \left| \Cl_2^*(\mathcal{O}) \right|-\left| \mathcal{I}_2^*(\mathcal{O}) \right|}{\left(\sum\limits_{0 \le b <n} \Vol(U_{1,b}^{r_2}(\mathbb{R}))_{<X}) \right) \prod\limits_p \Vol(S_p)} +o(1). \\
\intertext{Now, by the preceding sections, we know that this sum is equal to:}
&= \frac{\sum\limits_{0 \le b <n} \sum\limits_{\delta \in \mathcal{T}(r_2)} \sum\limits_{A \in \mathscr{L}_{\mathbb{Z}}} N_H(\mathcal{V}(\Lambda_{A,b}^\delta),X)}{\left(\sum\limits_{0 \le b <n} \Vol(U_{1,b}^{r_2}(\mathbb{R}))_{<X}) \right) \prod\limits_p \Vol(S_p)}. \\
\intertext{Expanding, we find:}
&= \sum_{\delta \in \mathcal{T}(r_2)} \sum_{A \in \mathscr{L}_{\mathbb{Z}}} \frac{2}{\sigma(r_2)} \Vol(\mathcal{F}_A^\delta) \prod_p \Vol(\SO_A(\mathbb{Z}_p)) \prod_{p\neq 2} m_p(A) \frac{\int_{f \in S_2} m_2(f,A) df}{\Vol(S_2)}. \\
\intertext{The indicator functions come into play at this point.}
&= \sum_{\delta \in \mathcal{T}(r_2)} \sum_{A \in \mathscr{L}_{\mathbb{Z}}} \frac{2}{\sigma(r_2)}  \chi_A(\delta) \Vol(\mathcal{F}_A)\prod_p \Vol(\SO_A(\mathbb{Z}_p)) \frac{\int_{f \in S_2} m_2(f,A) df}{\Vol(S_2)}\\
\intertext{We now break up the collection $\mathscr{L}_\mathbb{Z}$ into genera and sum over the forms in each genus separately before summing over the distinct genera. Since, both the characteristic function and the $p$-adic masses are constant over the forms in a single genus, they factor out of the inner sum.}
&= \sum_{\delta \in \mathcal{T}(r_2)} \sum_{\mathcal{G} \in \mathcal{G}_{\mathbb{Z}}} \sum_{A \in \mathcal{G} \cap \mathscr{L}_{\mathbb{Z}}} \frac{2}{\sigma(r_2)}  \chi_A(\delta) \Vol(\mathcal{F}_A)\prod_p \Vol(\SO_A(\mathbb{Z}_p))  \frac{\int_{f \in S_2} m_2(f,A) df}{\Vol(S_2)}\\
&= \sum_{\delta \in \mathcal{T}(r_2)} \sum_{\mathcal{G} \in \mathcal{G}_{\mathbb{Z}}} \frac{2}{\sigma(r_2)} \chi_{\mathcal{G}}(\delta) \frac{\int_{f \in S_2} m_2(f,\mathcal{G}) df}{\Vol(S_2)} \left(\sum_{A \in \mathcal{G} \cap \mathscr{L}_{\mathbb{Z}}}  \Vol(\mathcal{F}_A)\prod_p \Vol(\SO_A(\mathbb{Z}_p)) \right)\\
\intertext{Now, the inner sum gives the Tamagawa number of the special orthogonal group of an integral form in a genus. It is known to always be equal $2$, see for instance \cite{LanglandsFundamentalVolume} and \cite{EskinRudnickSarnak}. We denote it by $\tau(\SO)$.}
&= \tau(\SO) \sum_{\delta \in \mathcal{T}(r_2)} \sum_{\mathcal{G} \in \mathcal{G}_{\mathbb{Z}}}  \frac{2}{\sigma(r_2)} \chi_{\mathcal{G}}(\delta) \frac{\int_{f \in S_2} m_2(f,\mathcal{G}) df}{\Vol(S_2)}\\
\intertext{At this point, we simplify the sum using the fact that $\sigma(r_2) = \frac{1}{2^{r_2}}$, the value of the $2$-adic mass, the value of the infinite mass, and the classification of genera of unimodular integral quadratic forms as it appears in \cite{CasselsRationalQuadraticForms} or \cite{ConwaySloaneClassification}.}
&= \frac{\tau(\SO)}{2^{r_2-1}} \sum_{\delta \in \mathcal{T}(r_2)} \sum_{\mathcal{G} \in \mathcal{G}_{\mathbb{Z}}} \chi_{\mathcal{G}}(\delta) \frac{\int_{f \in S_2} m_2(f,\mathcal{G}) df}{\Vol(S_2)} \\
&= \frac{\tau(\SO)}{2^{r_2-1}} \sum_{\mathcal{G} \in \mathcal{G}_{\mathbb{Z}}} \sum_{\delta \in \mathcal{T}(r_2)}  \chi_{\mathcal{G}}(\delta) \frac{\int_{f \in S_2} m_2(f,\mathcal{G}) df}{\Vol(S_2)}\\
&= \frac{\tau(\SO)}{2^{r_2-1}} \sum_{\mathcal{G} \in \mathcal{G}_{\mathbb{Z}}}  \left( \frac{\int_{f \in S_2} m_2(f,\mathcal{G}) df}{\Vol(S_2)} \sum_{\delta \in \mathcal{T}(r_2)}  \chi_{\mathcal{G}}(\delta) \right)  \\
&= \frac{\tau(\SO)}{2^{r_2-1}} \left( c_2(\mathfrak{M}_1)\sum_{k \ge 0} {n \choose 4k}  + c_2(\mathfrak{M}_{-1})\sum_{k \ge 0} {n \choose 4k+2}  \right) \\
&= \frac{\tau(\SO)}{2^{r_2-1}} \Big( c_2(\mathfrak{M}_1)c_{\infty,0}  + c_2(\mathfrak{M}_{-1})c_{\infty,2}  \Big) \\ 
\intertext{Now, these values being known from previous sections, we substitute them and simplify.}
&= \frac{1}{2^{r_2-1}} \cdot 4 \left(2^{n-2} + 2^{\frac{n-2}{2}}\right) \\
&= \frac{1}{2^{r_2-1}} \cdot 4 \left( 2^{2r_2-2} + 2^{r_2-1}\right) \\
&= 2(2^{r_2} + 2)
\end{align*}

Therefore, since there are no units of negative norm for totally imaginary fields, the $2$ torsion in the oriented class group is twice as large as the $2$ torsion in the class group, we find that the average number of $2$ torsion elements in the class group of totally imaginary fields of even degree at least $4$ is: 
$$\boxed{{\rm Avg}(\Cl_2,\textrm{totally imaginary}) = 1+\frac{3}{2^{r_2}}} \quad . \\$$

\subsection{Oriented class group averages for non-evenly ramified not totally imaginary fields}

The computation for $r_1 > 0$ proceeds in a similar way.

\begin{align*}
&\frac{\sum\limits_{\substack{\mathcal{O} \in \mathfrak{R}, \\ H(\mathcal{O}) < X}}  \left| \mathcal{H}^*(\mathcal{O}) \right|-\left| \mathcal{I}_2^*(\mathcal{O}) \right|}{\left(\sum\limits_{0 \le b <n} \Vol(U_{1,b}^{r_2}(\mathbb{R}))_{<X}) \right) \prod\limits_p \Vol(S_p)} +o(1). \\
\intertext{Now, by the preceding sections, we know that this sum is equal to:}
&= \frac{\sum\limits_{0 \le b <n} \sum\limits_{\delta \in \mathcal{T}(r_2)} \sum\limits_{A \in \mathscr{L}_{\mathbb{Z}}} N_H(\mathcal{V}(\Lambda_{A,b}^\delta),X)}{\left(\sum\limits_{0 \le b <n} \Vol(U_{1,b}^{r_2}(\mathbb{R}))_{<X}) \right) \prod\limits_p \Vol(S_p)}. \\
\intertext{Expanding, we find:}
&= \sum_{\delta \in \mathcal{T}(r_2)} \sum_{A \in \mathscr{L}_{\mathbb{Z}}} \frac{2}{\sigma(r_2)} \Vol(\mathcal{F}_A^\delta) \prod_p \Vol(\SO_A(\mathbb{Z}_p)) \prod_{p\neq 2} m_p(A) \frac{\int_{f \in S_2} m_2(f,A) df}{\Vol(S_2)}. \\
\intertext{The indicator functions come into play at this point.}
&= \sum_{\delta \in \mathcal{T}(r_2)} \sum_{A \in \mathscr{L}_{\mathbb{Z}}} \frac{2}{\sigma(r_2)}  \chi_A(\delta) \Vol(\mathcal{F}_A)\prod_p \Vol(\SO_A(\mathbb{Z}_p)) \frac{\int_{f \in S_2} m_2(f,A) df}{\Vol(S_2)}\\
\intertext{We now break up the collection $\mathscr{L}_\mathbb{Z}$ into genera and sum over the forms in each genus separately before summing over the distinct genera. Since, both the characteristic function and the $p$-adic masses are constant over the forms in a single genus, they factor out of the inner sum.}
&= \sum_{\delta \in \mathcal{T}(r_2)} \sum_{\mathcal{G} \in \mathcal{G}_{\mathbb{Z}}} \sum_{A \in \mathcal{G} \cap \mathscr{L}_{\mathbb{Z}}} \frac{2}{\sigma(r_2)}  \chi_A(\delta) \Vol(\mathcal{F}_A)\prod_p \Vol(\SO_A(\mathbb{Z}_p))  \frac{\int_{f \in S_2} m_2(f,A) df}{\Vol(S_2)}\\
&= \sum_{\delta \in \mathcal{T}(r_2)} \sum_{\mathcal{G} \in \mathcal{G}_{\mathbb{Z}}} \frac{2}{\sigma(r_2)} \chi_{\mathcal{G}}(\delta) \frac{\int_{f \in S_2} m_2(f,\mathcal{G}) df}{\Vol(S_2)} \left(\sum_{A \in \mathcal{G} \cap \mathscr{L}_{\mathbb{Z}}}  \Vol(\mathcal{F}_A)\prod_p \Vol(\SO_A(\mathbb{Z}_p)) \right)\\
\intertext{Now, the inner sum gives the Tamagawa number of the special orthogonal group of an integral form in a genus. It is known to always be equal $2$, see for instance \cite{LanglandsFundamentalVolume} and \cite{EskinRudnickSarnak}. We denote it by $\tau(\SO)$.}
&= \tau(\SO) \sum_{\delta \in \mathcal{T}(r_2)} \sum_{\mathcal{G} \in \mathcal{G}_{\mathbb{Z}}}  \frac{2}{\sigma(r_2)} \chi_{\mathcal{G}}(\delta) \frac{\int_{f \in S_2} m_2(f,\mathcal{G}) df}{\Vol(S_2)}\\
\intertext{At this point, we simplify the sum using the fact that $\sigma(r_2) = \frac{1}{2^{r_1+r_2-1}}$, the value of the $2$-adic mass, the value of the infinite mass, and the classification of genera of unimodular integral quadratic forms as it appears in \cite{CasselsRationalQuadraticForms} or \cite{ConwaySloaneClassification}.}
&= \frac{2\tau(\SO)}{2^{r_1+r_2-1}} \sum_{\delta \in \mathcal{T}(r_2)} \sum_{\mathcal{G} \in \mathcal{G}_{\mathbb{Z}}} \chi_{\mathcal{G}}(\delta) \frac{\int_{f \in S_2} m_2(f,\mathcal{G}) df}{\Vol(S_2)} \\
&= \frac{\tau(\SO)}{2^{r_1+r_2-2}} \sum_{\mathcal{G} \in \mathcal{G}_{\mathbb{Z}}} \sum_{\delta \in \mathcal{T}(r_2)}  \chi_{\mathcal{G}}(\delta) \frac{\int_{f \in S_2} m_2(f,\mathcal{G}) df}{\Vol(S_2)}\\
&= \frac{\tau(\SO)}{2^{r_1+r_2-2}} \sum_{\mathcal{G} \in \mathcal{G}_{\mathbb{Z}}}  \left( \frac{\int_{f \in S_2} m_2(f,\mathcal{G}) df}{\Vol(S_2)} \sum_{\delta \in \mathcal{T}(r_2)}  \chi_{\mathcal{G}}(\delta) \right)  \\
&= \frac{\tau(\SO)}{2^{r_1+r_2-2}} \left( c_2(\mathfrak{M}_1)\sum_{k \ge 0} {n \choose 4k}  + c_2(\mathfrak{M}_{-1})\sum_{k \ge 0} {n \choose 4k+2}  \right) \\
&= \frac{\tau(\SO)}{2^{r_1+r_2-2}} \Big( c_2(\mathfrak{M}_1)c_{\infty,0}  + c_2(\mathfrak{M}_{-1})c_{\infty,2}  \Big) \\ 
\intertext{Now, these values being known from previous sections, we substitute them and simplify.}
&= \frac{1}{2^{r_1+r_2-2}} \cdot 2 \left((2^{n-2} \pm_8 2^{\frac{n-2}{2}})(2^{r_1-2} \pm_8 2^{\frac{r_1-2}{2}})+(2^{n-2} \mp_8 2^{\frac{n-2}{2}})(2^{r_1-2} \mp_8 2^{\frac{r_1-2}{2}})\right) \\
&= \frac{1}{2^{r_1+r_2-2}} \cdot 2 \cdot 2 \left( 2^{n+r_1-4} + 2^{\frac{n+r_1-4}{2}}\right) \\
&= 8\left(2^{r_1+r_2-3} + 2^{-1} \right)\\
&= 2^{r_1+r_2} + 4
\end{align*}

Therefore, this leads to the following formula for average $2$-torsion in the oriented class group of non-evenly ramified fields with at least one real embedding: $$\boxed{{\rm Avg}(\Cl_2^*) = 1 + \frac{3}{2^{r_1+r_2-1}}} \quad .$$

%\subsection{Some corollaries about parity and units}

%\begin{cor}
%We can deduce the following corollaries from the fact that the non-evenly ramified averages above are small. Indeed, recall that the oriented class group is isomorphic to the ordinary class group if there is a unit of negative norm and is an extension of the ordinary class group by $\mathbb{Z}/2\mathbb{Z}$ otherwise. In particular, if the oriented class number is odd, the class number is equal to the oriented class number, and the field's ordinary class number must be odd. 
%
%Let $r_1 = 0$, then: 
%\begin{enumerate}
%\item A positive proportion of fields (at least $1-\frac{3}{2^{r_2}}$ of non-evenly ramified fields) have odd class number.
%\end{enumerate}
%
%Let $r_1 > 0$, then:
%\begin{enumerate}
%\item A positive proportion (at least $1-\frac{3}{2^{r_1+r_2-1}}$ of non-evenly ramified fields) of fields have odd class number.
%\item A positive proportion (at least $1-\frac{3}{2^{r_1+r_2-1}}$ of non-evenly ramified fields) of fields have odd oriented class number.
%\item A positive proportion (at least $1-\frac{3}{2^{r_1+r_2-1}}$ of non-evenly ramified fields) have a unit of negative norm. 
%\end{enumerate}
%\end{cor}

\subsection{General oriented class group averages}

For simplicity, we carried out the calculations of the previous subsections under the assumption that there was no even ramification. In general, even ramification increases the mass at $p \neq 2$ from $1$ to $2$. If we take this into account, the computations of the previous section carry through with an additional term which captures this increase in total mass.

\begin{defn}
Let $\mathfrak{R} \subset \mathfrak{R}^{r_1,r_2}$ be a family of rings (unramified at $2$) corresponding to an acceptable family of local specifications $\Sigma = (\Sigma_p)_p$. We define the \emph{even ramification density} at $p$, $r_p(\mathfrak{R})$, as the density in $\Sigma_p$ of elements of $\Sigma_p$ which are evenly ramified at $p$.
\end{defn}

We obtain the following averages by calculating as above.

\begin{thm}[General Averages for the oriented class group]
Let $\mathfrak{R} \subset \mathfrak{R}^{r_1,r_2}$ be a family of rings (unramified at $2$) corresponding to an acceptable family of local specifications $\Sigma = (\Sigma_p)_p$ and let $r_p(\mathfrak{R})$ denotes its even ramification density at $p$. 

If $r_1 = 0$, the average number of $2$-torsion elements in the oriented class group over $\mathfrak{R}$ is given by: 
$$\boxed{{\rm Avg}(\Cl_2^*,\mathfrak{R}) = 2\prod_{p \neq 2} (1+r_p(\mathfrak{R}))\left(1+\frac{2}{2^{r_2}}\right) + \frac{2}{2^{r_2}}} \quad . \\$$

If $r_1 > 0$, the average number of $2$-torsion elements in the oriented class group over $\mathfrak{R}$ is given by: 
$$\boxed{{\rm Avg}(\Cl_2^*,\mathfrak{R}) = \prod_{p \neq 2} (1+r_p(\mathfrak{R}))\left(1+\frac{2}{2^{r_1+r_2-1}}\right) + \frac{1}{2^{r_1+r_2-1}}} \quad . \\$$
\end{thm}

\begin{rem}
A priori, we have to be careful about summing over the $b$ because the even ramification densities at different values of $b$ might be different. However, this occurs only at primes $p$ which divide $n$. For those finitely many primes, the quantity $\sum_{0\le b \le n-1} \prod_{p|n}(1+r_{p,b}(\mathfrak{R}))$ when expanded is a finite sum of densities, which can then be factored again to give $\prod_{p |n}(1+r_p(\mathfrak{R}))$. The same considerations apply to the final calculations in the next section. (Thanks to Ashvin Swaminathan and Arul Shankar for pointing out this issue and telling me how to fix it). 
\end{rem}

\section{Averages for the class group and narrow class group} \label{Averages for the class group and narrow class group}

We can handle the question of computing the average number of $2$-torsion elements in the usual class group by looking at $\SL_n^{\pm}$-orbits. The catch is that to get a torsor of the full class group, we must capture ideals with the property that $I^2 = (\alpha)$ with $\alpha \in K^\times$ and $N(\alpha) < 0$ in the rigid parametrisation. But those occur when we consider $-f$ instead of our monic $f$. Of course, this now introduces potential double-counting for rings that have a unit of negative norm. But this is offset by the fact that the fibre space $((R_f^\times)_{N \equiv 1})/((R_f^\times)^2)$ is half as large when $R_f$ has a unit of negative norm as opposed to when it does not. We therefore recover the result of \cite{MelanieWoodIdealClasses} that the set of pairs $(A,B)$ with resolvent polynomial equal to $\pm f$ is an extension by $R_f^\times/(R_f^\times)^2$ of a torsor of the $2$-torsion in the class group of $R_f$. \\

We explain the parametrisation in more detail.  

\subsection{Rigid parametrisation by pairs of symmetric bilinear forms}
 Let $T$ be a principal ideal domain. Recall once again the rigid parametrisation of the pairs of bilinear forms $(A,B) \in V(T)$ with resolvent polynomial equal to $f$ in terms of the based fractional ideal data for $R_f$.

\begin{thm}[\cite{MelanieWoodIdealClasses}]
Take a non-degenerate binary $n$-ic form $f \in U(T)$ and let $R_f = \frac{T[x]}{(f(x))}$. Then the pairs symmetric bilinear forms $(A,B) \in V(T)$ with $f_{(A,B)} = f$ are in bijection with equivalence classes of triples $$(I, \mathcal{B},\delta)$$ where $I \subset K_f$ is a based fractional ideal of $R_f$ with basis $\mathcal{B}$ given by a $T$ module isomorphism $\mathcal{B} \colon I \rightarrow T^n$, $\delta \in K_f^\times$ such that $I^2 \subset \delta R_f^{n-3}$ as ideals and the norm equation holds $N(I)^2 = N(\delta) N(R_f^{n-3})$ (as based ideals). Two triples $(I,\mathcal{B},\delta)$ and $(I',\mathcal{B}',\delta')$ are equivalent if there exists a $\kappa \in K_f^\times$ with the property that $I = \kappa I'$, $\mathcal{B} \circ (\times \kappa) = \mathcal{B}'$ and $\delta = \kappa^2 \delta$.
\end{thm}

\subsection{$\SL_n^\pm(\mathbb{Z})$-orbits and the class group}

We now let $\SL_n^\pm(T)$ act on $V$ by change of basis. This action corresponds, at the level of equivalence classes of triples $(I,\mathcal{B},\delta)$, to the action of $\SL_n^\pm(T)$ on the basis $\mathcal{B}$ of the fractional ideal $I$. This action only changes the basis and does not change $I$ nor $\delta$. In particular, it does not change the defining conditions: 
\begin{enumerate}
\item $I^2 \subset \delta R_f^{n-3}$
\item $N(I)^2 = N(\delta)N(R_f^{n-3})$
\end{enumerate} 
Indeed, acting by an element of $\SL_n^{\pm}(T)$ with negative determinant reverses the orientation of the ideal $I$, but since the second condition only depends on $N(I)^2$, this is immaterial.\\

The rigid parametrisation can be reformulated to describe $\SL_n^\pm(\mathbb{Z})$ orbits.  
\begin{thm}
Let $f$ be a non-degenerate monic binary $n$-ic form $f \in U_1(\mathbb{Z})$, let $R_f = \frac{\mathbb{Z}[x]}{(f(x))}$ and $K_f = R_f \otimes_{\mathbb{Z}} \mathbb{Q}$. The $\SL_n^\pm(\mathbb{Z})$ orbits on pairs symmetric bilinear forms $(A,B) \in V(\mathbb{Z})$ with $f_{(A,B)} = f$ are in bijection with equivalence classes of pairs $$(I,\delta)$$ where $I \subset K_f$ is a fractional ideal of $R_f$, and $\delta \in K_f^\times$ is such that $I^2 \subset \delta R_f$ as ideals and the norm equation $N(I)^2 = N(\delta)$ holds (this is the norm of $\delta$ as an ideal). Two pairs $(I,\delta)$ and $(I',\delta')$ are equivalent if there exists a $\kappa \in K_f^\times$ with the property that $I = \kappa I'$ and $\delta = \kappa^2 \delta'$.
\end{thm}

We now describe the relation between $\SL_n^\pm(\mathbb{Z})$ orbits on $\pi^{-1}(\pm f) \subset V(\mathbb{Z})$ and the $2$-torsion part of the class group of $R_f$. 

\begin{defn} 
A pair $(I,\delta)$ is said to be projective if the ideal $I$ is invertible. Equivalently, a pair $(I,\delta)$ is projective if $I^2 = (\delta)$. For an $S_n$ order $\mathcal{O}$, we write $H(\mathcal{O})$ for the set of oriented projective pairs on $\pi^{-1}(\pm f) \subset V(\mathbb{Z})$. 
\end{defn}

As the oriented class group is comprised of invertible oriented ideals, let us consider the restriction of the parametrisation above to the set $H(\mathcal{O})$. Consider the forgetful map:
\begin{equation*}
H(\mathcal{O}) \longrightarrow \Cl_2(\mathcal{O}).
\end{equation*}
Elements $(I_0,\delta_0)$ in the kernel of this forgetful map are such that $I_0^2 = (\delta)$ and $N(\delta_0) = \pm N(I_0)^2$ and have the property that $I_0 = (\alpha)$ for some $\alpha$ in $K$ as ideals. Thus, $(I_0,\delta_0) \sim ((\alpha),\delta_0) \sim (\mathcal{O},\alpha^{-2}\delta_0)$. Now, $\mathcal{O} \subset \alpha^{-2}\delta_0 \mathcal{O}$ and $\pm 1 = N(\alpha^{-2} \delta_0)$. This implies that $(I_0,\delta_0)$ is equivalent to $(\mathcal{O},u)$ where $u$ is norm $\pm 1$ unit of $\mathcal{O}^{\times}$. We are allowed to mod out $u$ by squares of elements of $K^\times$ which fix the ideal $\mathcal{O}$. But those are precisely the units of $\mathcal{O}$. The sequence above can thus be completed to a short exact sequence: 
\begin{equation*}
1 \longrightarrow \frac{\mathcal{O}^{\times}}{(\mathcal{O}^{\times})^2} \longrightarrow H(\mathcal{O}) \longrightarrow \Cl_2(\mathcal{O}) \longrightarrow 1.
\end{equation*}

Applying Dirichlet's unit theorem allows us to compute that $100\%$ of the time (since having non-trivial torsion is rare):
\begin{equation*} 
\left|\frac{\mathcal{O}^{\times}}{(\mathcal{O}^{\times})^2} \right| = 2^{r_1+r_2}.
\end{equation*} 

We thus obtain the following formula for the number of elements in $H(\mathcal{O})$.

\begin{lem}
Let $\mathcal{O}$ be an order in an $S_n$-number field of degree $n$ and signature $(r_1,r_2)$. Then: 
\begin{equation*}
\left|H(\mathcal{O})\right| = 2^{r_1+r_2} \left|\Cl_2(\mathcal{O}) \right|. 
\end{equation*}
\end{lem}

We can also say something about the narrow class group in this context. The proof that the fibres have the right size is exactly the same as in \cite{HoShankarVarmaOdd}. 

\begin{lem}
if $H^+(\mathcal{O})$ denotes the subgroup of $H(\mathcal{O})$ consisting of pairs $(I,\delta)$ such that $\delta$ is positive under every real embedding of the fraction field of $\mathcal{O}$, then $$|H^+(\mathcal{O})| = 2^{r_2} |\Cl_2^+(\mathcal{O})|.$$
\end{lem}

We now proceed to the description of the points in the cuspidal regions. 

\begin{lem}[Even degree distinguished orbits lemma] 
$(A,B) \in V(\mathbb{Q})$ is distinguished if and only if there is an $\SL_n^\pm(\mathbb{Q})$ translate of $(A,B)$ with the property that $$a_{i,j} = b_{i,j} = 0$$ for all $1 \le i,j \le \frac{n}{2}$ except for $i = j = \frac{n}{2}$ for which $a_{\frac{n}{2}\, \frac{n}{2}} = 0 \neq b_{\frac{n}{2}\, \frac{n}{2}}$.
\end{lem}

\begin{defn}
Let $\mathcal{O}$ be an order. We denote by $\mathcal{I}_2(\mathcal{O})$ the $2$-torsion subgroup of the ideal group of $\mathcal{O}$. 
\end{defn} 

\begin{prop}
Let $\mathcal{O}_f$ be an order corresponding to the integral primitive irreducible non-degenerate monic binary form $f$. Then, $I_2(\mathcal{O}_f)$ is in natural bijection with the set of projective reducible $\SL_n^\pm(\mathbb{Z})$-orbits on $V(\mathbb{Z}) \cap \pi^{-1}(f)$. 
\end{prop}

\subsection*{$\SL_n$-orbits over fields and local rings} In this section, we compute the number orbits, and the size of the stabilisers for the action of $\SL_n$ on the arithmetic rings $\mathbb{Z}_p$, $\mathbb{Q}$, and $ \mathbb{R}$. Let $T$ be a principal ideal domain. We first restate the rigid parametrisation in this case.
\begin{thm}[\cite{BhargavaGrossWang}]
Let $f$ be a non-degenerate monic binary $n$-ic form $f \in U_1(T)$, let $R_f = \frac{T[x]}{(f(x))}$ and $K_f = R_f \otimes_{\mathbb{Z}} \mathbb{Q}$. The $\SL_n^\pm(T)$ orbits on pairs symmetric bilinear forms $(A,B) \in V(T)$ with $f_{(A,B)} = f$ are in bijection with equivalence classes of pairs $$(I,\delta)$$ where $I \subset K_f$ is a fractional ideal of $R_f$ and $\delta \in K_f^\times$ such that $I^2 \subset \delta R_f$ as ideals and the norm equation $N(I)^2 = N(\delta)$ holds. Two pairs $(I,\delta)$ and $(I',\delta')$ are equivalent if there exists a $\kappa \in K_f^\times$ with the property that $I = \kappa I'$ and $\delta = \kappa^2 \delta'$.
\end{thm}

Then, we have the following theorem whose proof is straightforward.

\begin{lem}
The stabiliser in $\SL_n^\pm(T)$ corresponds to the $2$-torsion of $R_f^\times$: $$R_f^\times[2].$$ 
\end{lem}

\begin{rem}
It follows that the $\SL_n^\pm(\mathbb{Q})$ stabiliser of any element $v$ whose resolvent is irreducible is equal to $2$. 
\end{rem}

Just as in \cite{HoShankarVarmaOdd}, we can compute the number of orbits with the caveat that we need to distinguish the case where $f$ has an odd degree factor from the case where all the factors of $f$ are even. 

\begin{lem}
Let $T$ be a field or $\mathbb{Z}_p$. Let $f$ be a monic separable, non-degenerate binary form in $U_1(T)$. Then the projective $\SL_n(T)$ orbits of $V(T)$ with resolvent $\pm f$ are in bijection with elements of
\begin{equation*}
(R_f^\times/(R_f^\times)^2)_{N \equiv \pm 1} .
\end{equation*}
\end{lem}

\begin{rem} 
We now describe the real orbits in the case where the leading coefficient of $f$ is $1$ and in the case where the leading coefficient is $-1$. Suppose that $f$ is a non-degenerate polynomial of degree $n$ with $r_1$ real roots and $2r_2$ complex roots. First, suppose that the leading coefficient of $f$ is $-1$. If $r_1 = 0$ there are no real orbits while if $r_1 > 0$, there are $2^{r_1-1}$ orbits and the stabiliser has size $2^{r_1+r_2}$. Now, suppose that the leading coefficient of $f$ is $1$. If $r_1 = 0$, there is $1$ real orbit while if $r_1 > 0$, there are $2^{r_1-1}$ real orbits. In both cases, the stabiliser has size $2^{r_1+r_2}$. 
\end{rem}

Now, instead of considering $\SO_A$ orbits, we think $\Or_A$ orbits. Then the reduction theory, the cusp cutoff, the sieve, as well as the change of variable formula carry through precisely as above. We, therefore, obtain the following proposition. 

\begin{prop} We have $$\Vol\left( \mathcal{F}_A \cdot R_{A}^{r_2,\delta} (X) \right) =  \chi_A(\delta) \left| \mathcal{J}_A \right| \Vol( \mathcal{F}_A) \Vol(U(\mathbb{R})^{r_2}_{H < X}).$$ Let $S_p \subset U_{1,b}(\mathbb{Z}_p)$ be a closed subset whose boundary has measure $0$. Consider the set $\Lambda(A)_p = V_{A,b}(\mathbb{Z}_p) \cap \pi^{-1}(S_p)$. Then we have $$\Vol( \Lambda_{A}(p)) = \left| \mathcal{J}_A \right|_p \Vol(\Or_A(\mathbb{Z}_p)) \int_{f \in S_p} m_p(f,A) \, df$$
where $$m_p(f,A) := \sum_{v \in \frac{V_{A,b}(\mathbb{Z}_p) \cap \pi^{-1}(f)}{\Or_A(\mathbb{Z}_p)}} \frac{1}{\# {\rm Stab}_{\Or_A(\mathbb{Z}_p)}(v)} .$$
\end{prop}

The computation of the total local mass could be slightly more delicate since we need to differentiate the case of $f$ having leading coefficient $\pm 1$. We treat this in detail in the next subsection.

\subsection{The product of local volumes and the local mass}

The total mass for the case where $f$ has leading coefficient $1$ is the following. 

\begin{lem}[$+1$ total mass]
Let $R$ be a non-degenerate ring of degree $n$ over $\mathbb{Z}_p$. Let $N$ denote the norm map from $N \colon R^\times/(R^\times)^2 \rightarrow \mathbb{Z}_p^\times/ (\mathbb{Z}_p^\times)^2$. The quantity
\begin{equation*}
\frac{\left| (R^\times/(R^\times)^2)_{N \equiv 1} \right|}{\left|R^\times[2] \right|} 
\end{equation*}
is equal to $\frac{1}{\left| N(R^\times) \right|}$ if $p \neq 2$ and to $\frac{2^n}{\left| N(R^\times) \right|}$ if $p = 2$. 
\end{lem}

The total mass for the case where $f$ has leading coefficient $-1$ is the following. 

\begin{lem}[$-1$ total mass]
Let $R$ be a non-degenerate ring of degree $n$ over $\mathbb{Z}_p$. Let $N$ denote the norm map from $N \colon R^\times/(R^\times)^2 \rightarrow \mathbb{Z}_p^\times/ (\mathbb{Z}_p^\times)^2$. The quantity
\begin{equation*}
\frac{\left| (R^\times/(R^\times)^2)_{N \equiv -1} \right|}{\left|R^\times[2] \right|} 
\end{equation*}
is equal to $0$ if $-1$ is not in the image of $N$. Otherwise, it is equal to $\frac{1}{\left| N(R^\times) \right|}$ if $p \neq 2$ and to $\frac{2^n}{\left| N(R^\times) \right|}$ if $p = 2$. 
\end{lem}
\begin{proof}
The fibres of a group homorphism are either empty or are of the same size as the kernel. The size of the kernel was calculated in the previous lemma. 
\end{proof}

In particular, we get the following statement concerning the total local masses for maximal rings which not evenly ramified at $2$.

\begin{lem}
The total local mass for $f \in \mathbb{Z}_p[x]$ maximal and not evenly ramified is given by: 
\begin{equation*}
m_p^\pm(f) = \begin{cases}
2^{n-2} & \text{ if } p = 2 \\
\frac{1}{2} & \text{ if } p \neq 2 
\end{cases}.
\end{equation*}
The total local mass for $f \in \mathbb{Z}_p[x]$, $p \neq 2$, maximal and evenly ramified is given by: 
\begin{equation*}
m_p^+(f) = 1
\end{equation*}
\begin{equation*}
m_p^-(f) = \begin{cases}
1 & \text{ if } p \equiv 1 \mod 4 \\
0 & \text{ if } p \equiv 3 \mod 4
\end{cases}.
\end{equation*}
\end{lem}
\begin{proof}
$-1$ is not in the image of the norm map when $f$ is maximal and evenly ramified at $p \equiv 3 \mod 4$. 
\end{proof}

\subsection{Distribution of total local masses among genera}

The distribution techniques work well here. As soon as the total mass is known, knowing the local mass for each genus is straightforward easy. We note the results here. \\ 

Let $p \neq 2$. Then, up to $\SL_n^\pm(\mathbb{Z}_p)$ equivalence, there is a unique bilinear form of determinant $1$ and a unique bilinear form of determinant $-1$. Both have Hasse-Witt symbol equal to $1$. Therefore, the local masses are the same as before.\\

The computation of the local masses at $p =2,\infty$ is more delicate and is treated in the next to subsections.

\subsection{Point count and the $2$-adic masses} 

For $p = 2$, up to $\SL_n^\pm(\mathbb{Z}_2)$ equivalence, there are $3$ quadratic forms of determinant $1$ and $2$ quadratic forms of determinant $-1$. By restricting to rings which are not evenly ramified at $2$, we eliminate one of the forms of determinant $1$.

For $n \equiv 0 \mod 4$ we can take:  
$$\mathfrak{M}^+_1 = \begin{pmatrix}
1 &   &  &  & &  \\ 
 & \ddots &  &  & & \\ 
 &  & 1 &  &  &\\ 
 &  &  & 1 &  &\\ 
 &  &  &  & 1 &\\
 &  &  &  &  & 1\\
\end{pmatrix}, \,\,\,\,
\mathfrak{M}^+_{-1} = \begin{pmatrix}
1 &   &  &  & &  \\ 
 & \ddots &  &  & & \\ 
 &  & 1 &  &  &\\ 
 &  &  & 1 &  &\\ 
 &  &  &  & -1 &\\
 &  &  &  &  & -1\\
\end{pmatrix}.$$
$$\mathfrak{M}^-_1 = \begin{pmatrix}
1 &   &  &  & &  \\ 
 & \ddots &  &  & & \\ 
 &  & 1 &  &  &\\ 
 &  &  & 1 &  &\\ 
 &  &  &  & 1 &\\
 &  &  &  &  & -1\\
\end{pmatrix}, \,\,\,\,
\mathfrak{M}^-_{-1} = \begin{pmatrix}
1 &   &  &  & &  \\ 
 & \ddots &  &  & & \\ 
 &  & 1 &  &  &\\ 
 &  &  & -1 &  &\\ 
 &  &  &  & -1 &\\
 &  &  &  &  & -1\\
\end{pmatrix}.$$

For $n \equiv 2 \mod 4$ we can take: 
$$\mathfrak{M}^+_1 = \begin{pmatrix}
1 &   &  &  & &  \\ 
 & \ddots &  &  & & \\ 
 &  & 1 &  &  &\\ 
 &  &  & 1 &  &\\ 
 &  &  &  & 1 &\\
 &  &  &  &  & -1\\
\end{pmatrix}, \,\,\,\,
\mathfrak{M}^+_{-1} = \begin{pmatrix}
1 &   &  &  & &  \\ 
 & \ddots &  &  & & \\ 
 &  & 1 &  &  &\\ 
 &  &  & -1 &  &\\ 
 &  &  &  & -1 &\\
 &  &  &  &  & -1\\
\end{pmatrix}.$$
$$\mathfrak{M}^-_1 = \begin{pmatrix}
1 &   &  &  & &  \\ 
 & \ddots &  &  & & \\ 
 &  & 1 &  &  &\\ 
 &  &  & 1 &  &\\ 
 &  &  &  & 1 &\\
 &  &  &  &  & 1\\
\end{pmatrix}, \,\,\,\,
\mathfrak{M}^-_{-1} = \begin{pmatrix}
1 &   &  &  & &  \\ 
 & \ddots &  &  & & \\ 
 &  & 1 &  &  &\\ 
 &  &  & 1 &  &\\ 
 &  &  &  & -1 &\\
 &  &  &  &  & -1\\
\end{pmatrix}.$$

The superscripts indicate the sign of $(-1)^\frac{n}{2} \det(\cdot)$ and the subscripts indicate the Hasse-Witt symbol.

Now, the argument proceeds just as before except that now the \emph{octane values} of $\mathfrak{M}^-_1$ and $\mathfrak{M}^-_{-1}$ are now $\pm 2 \mod 8$. This means that the volumes of the respective orthogonal groups are equal and thus that the masses at $2$ are equal! We summarise the numbers in the following lemma. 

\begin{cor}
The $2$-adic masses are: 
\begin{alignat*}{2}
&c_2(n,\mathfrak{M}^+_1) &&= \frac{1}{2}\left(2^{n-2} \pm_8 2^{\frac{n-2}{2}} \right) \Vol(S_2)\\
&c_2(n,\mathfrak{M}^+_{-1}) &&= \frac{1}{2}\left(2^{n-2} \mp_8 2^{\frac{n-2}{2}} \right) \Vol(S_2) \\
&c_2(n,\mathfrak{M}^-_1) &&= \frac{1}{2} \left(2^{n-2} \right) \Vol(S_2) \\
&c_2(n,\mathfrak{M}^-_{-1}) &&= \frac{1}{2} \left(2^{n-2} \right) \Vol(S_2) 
\end{alignat*}
where $\pm_8$ is $+$ if $n$ is congruent to $0$ or $2 \mod 8$ and $-$ otherwise.
\end{cor}

\subsection{The infinite masses} We assume that $r_1 > 0$ since the totally imaginary case was already covered. Since the $2$-adic masses for leading coefficient $-1$ are equal across genera, we will only need to know the total infinite mass, which is $2^{r_1-1}$. 

\subsection{Statistical consequences} We now pool together the elements assembled in the previous subsections to compute the averages. \\

The computation basically carries through as above. In order to deal with the Tamagawa number and the half-integral local masses (which at first might seem to multiply to $0$), one needs to use the following facts. If $\mathcal{G}$ denotes a genus, then we have the following identity of local volumes:
\begin{equation*}
\Vol(\Or_A(\mathbb{Z}_p)) = 2 \Vol(\SO_A(\mathbb{Z}_p)),
\end{equation*}
%\begin{equation*}
% \Vol(\Or_A(\mathbb{Z}) \backslash  \Or_A(\mathbb{R})) =  \Vol(\SO_A(\mathbb{Z}) \backslash  \SO_A(\mathbb{R})),
%\end{equation*}
\begin{equation*}
\sum_{A \in \mathcal{G}} \Vol(\Or_A(\mathbb{Z}) \backslash  \Or_A(\mathbb{R}))\prod_p \Vol(\SO_A(\mathbb{Z}_p)) = 2, 
\end{equation*} 
where the sum is over all $\SL^\pm(\mathbb{Z})$ equivalence classes of integral representatives in $\mathcal{G}$.  \\

We now proceed with the computation. \\

\begin{defn}
Let $\mathfrak{R} \subset \mathfrak{R}^{r_1,r_2}$ be a family of rings (unramified at $2$) corresponding to an acceptable family of local specifications $\Sigma = (\Sigma_p)_p$. We define the \emph{even ramification density} at $p$, $r_p(\mathfrak{R})$, as the density in $\Sigma_p$ of elements of $\Sigma_p$ which are evenly ramified at $p$.
\end{defn} 

Let $\mathscr{L}^+_{\mathbb{Z}}$ denote a set of representatives of integral bilinear forms of determinant $1$ under the action of $\SL_n^{\pm}(\mathbb{Z})$ and let $\mathscr{L}^-_{\mathbb{Z}}$ denote a set of representatives of integral bilinear forms of determinant $-1$ under the action of $\SL_n^{\pm}(\mathbb{Z})$. Denote by $\mathcal{G}_{\mathbb{Z}}^+$ the set of genera of quadratic $n$-ary forms containing an integral element of determinant $1$ and denote by $\mathcal{G}_{\mathbb{Z}}^-$ the set of genera of quadratic $n$-ary forms containing an integral element of determinant $-1$. Notice that $\mathcal{G}_{\mathbb{Z}}^+$ partitions $\mathscr{L}_{\mathbb{Z}}^+$ and that $\mathcal{G}_{\mathbb{Z}}^-$ partitions $\mathscr{L}_{\mathbb{Z}}^-$.

\begin{align*}
&\frac{\sum\limits_{\substack{\mathcal{O} \in \mathfrak{R}, \\ H(\mathcal{O}) < X}}  \left| \mathcal{H}(\mathcal{O}) \right|-\left| \mathcal{I}_2(\mathcal{O}) \right|}{ \left(\sum\limits_{0 \le b <n} \Vol(U_{1,b}^{r_2}(\mathbb{R}))_{<X}) \right) \prod\limits_p \Vol(S_p)} +o(1). \\
\intertext{Now, by the preceding sections, we know that this sum is equal to:}
&= \frac{\sum\limits_{0 \le b <n} \sum\limits_{\delta \in \mathcal{T}(r_2)} \sum\limits_{A \in \mathscr{L}^-_{\mathbb{Z}}} N_H(\mathcal{V}(\Lambda_{A,b}^\delta),X) + \sum\limits_{0 \le b <n} \sum\limits_{\delta \in \mathcal{T}(r_2)} \sum\limits_{A \in \mathscr{L}^+_{\mathbb{Z}}} N_H(\mathcal{V}(\Lambda_{A,b}^\delta),X)}{\left(\sum\limits_{0 \le b <n} \Vol(U_{1,b}^{r_2}(\mathbb{R}))_{<X}) \right) \prod\limits_p \Vol(S_p)}. 
\end{align*}

Expanding, we find that the first sum is equal to: 
\begin{equation*}
 \sum_{\delta \in \mathcal{T}(r_2)} \sum_{A \in \mathscr{L}^-_{\mathbb{Z}}} \frac{2}{\sigma(r_2)} \Vol(\mathcal{F}_A^\delta) \prod_p \Vol(\SO_A(\mathbb{Z}_p)) \prod_{p \equiv 1 \textrm{ mod } 4} (1+r_p(\mathfrak{R})) \prod_{p \equiv 3 \textrm{ mod } 4} (1-r_p(\mathfrak{R})) \frac{\int_{f \in S_2} 2m_2(f,A) df}{\Vol(S_2)}.
\end{equation*}

Expanding, we find that the second sum is equal to: 
\begin{equation*}
 \sum_{\delta \in \mathcal{T}(r_2)} \sum_{A \in \mathscr{L}^+_{\mathbb{Z}}} \frac{2}{\sigma(r_2)} \Vol(\mathcal{F}_A^\delta) \prod_p \Vol(\SO_A(\mathbb{Z}_p)) \prod_{p \neq 2}  (1+r_p(\mathfrak{R})) \frac{\int_{f \in S_2} 2m_2(f,A) df}{\Vol(S_2)}.
\end{equation*}

We can now process both of these sums separately just as we did before by breaking up $\mathscr{L}_\mathbb{Z}^-$ and $\mathscr{L}_\mathbb{Z}^+$ into genera and then first summing within each genus and then across the different genera. \\

After this is done, the $\mathscr{L}_\mathbb{Z}^-$ sum becomes: 
\begin{align*}
& \frac{2\tau(\SO)}{2^{r_1+r_2}} \Big(2c_2(n,\mathfrak{M}^-_1)c_{\infty,0}  + 2c_2(n,\mathfrak{M}^-_{-1})c_{\infty,2}  \Big) \left( \prod_{p \equiv 1 \textrm{ mod } 4} (1+r_p(\mathfrak{R})) \prod_{p \equiv 3 \textrm{ mod } 4}  (1-r_p(\mathfrak{R}))  \right)\\
&= \frac{1}{2^{r_1+r_2-1}} \cdot 2 \cdot 2^{n-2} \cdot 2^{r_1-1} \cdot \left( \prod_{p \equiv 1 \textrm{ mod } 4} (1+r_p(\mathfrak{R})) \prod_{p \equiv 3 \textrm{ mod } 4}  (1-r_p(\mathfrak{R}))  \right) \\
&= 2^{r_1+r_2-1} \left( \prod_{p \equiv 1 \textrm{ mod } 4} (1+r_p(\mathfrak{R})) \prod_{p \equiv 3 \textrm{ mod } 4}  (1-r_p(\mathfrak{R}))  \right) \\
\end{align*}

And the $\mathscr{L}_\mathbb{Z}^+$ sum becomes: 
\begin{align*}
& \frac{2\tau(\SO)}{2^{r_1+r_2}} \Big(2c_2(n,\mathfrak{M}^+_1)c_{\infty,0}  + 2c_2(n,\mathfrak{M}^+_{-1})c_{\infty,2}  \Big) \left( \prod_{p \neq 2}  (1+r_p(\mathfrak{R})) \right) \\
&= \frac{1}{2^{r_1+r_2-1}} \cdot 2 \cdot 2 \left( 2^{n+r_1-4} + 2^{\frac{n+r_1-4}{2}}\right) \cdot \left( \prod_{p \neq 2}  (1+r_p(\mathfrak{R})) \right)  \\
&= \left(2^{r_1+r_2-1}+ 2 \right) \left( \prod_{p \neq 2}  (1+r_p(\mathfrak{R})) \right) \\
\end{align*}

Therefore, we find the following formula for average $2$-torsion in the class group of fields square-free at $2$ with at least one real embedding: $$\boxed{\begin{split}
{\rm Avg}(\Cl_2,\mathfrak{R}) &= \frac{1}{2}\prod_{p \equiv 1 \textrm{ mod } 4} (1+r_p(\mathfrak{R})) \left( \prod_{p \equiv 3 \textrm{ mod } 4}  (1-r_p(\mathfrak{R})) + \prod_{p \equiv 3 \textrm{ mod } 4}  (1+r_p(\mathfrak{R})) \right) \\ 
&+ \frac{1+2\prod_{p \neq 2} (1+r_p(\mathfrak{R}))}{2^{r_1+r_2}}
\end{split}
} \quad .$$

As a sanity check, note that $\left( \prod_{p \equiv 3 \textrm{ mod } 4}  (1-r_p(\mathfrak{R})) + \prod_{p \equiv 3 \textrm{ mod } 4}  (1+r_p(\mathfrak{R})) \right) \ge 2$, so the quantity above is always greater than $1$. \\ 

The computation for the narrow class group is similar and gives us: 
\begin{align*}
&\frac{\sum\limits_{\substack{\mathcal{O} \in \mathfrak{R}, \\ H(\mathcal{O}) < X}} 2^{r_2} \left| \Cl_2^+(\mathcal{O}) \right|-\left| \mathcal{I}_2(\mathcal{O}) \right|}{\left(\sum\limits_{0 \le b <n} \Vol(U_{1,b}^{r_2}(\mathbb{R}))_{<X}) \right) \prod\limits_p \Vol(S_p)} +o(1) \\ 
&= \frac{\sum\limits_{0 \le b <n}\sum\limits_{A \in \mathscr{L}^\pm_{\mathbb{Z}}} N_H(\mathcal{V}(\Lambda_{A,b}^{\delta_{\gg 0}}),X)}{\left(\sum\limits_{0 \le b <n} \Vol(U_{1,b}^{r_2}(\mathbb{R}))_{<X}) \right) \prod\limits_p \Vol(S_p)} \\
&= \sum_{A \in \mathscr{L}^\pm_{\mathbb{Z}}} \frac{2}{\sigma(r_2)} \Vol(\mathcal{F}_A^{\delta_{\gg 0}}) \prod_p \Vol(\SO_A(\mathbb{Z}_p))  \left( \prod_{p \neq 2}  (1+r_p(\mathfrak{R})) \right) \frac{\int_{f \in S_2} 2m_2(f,A) df}{\Vol(S_2)}\\
&=\sum_{A \in \mathscr{L}^\pm_{\mathbb{Z}}} \frac{2}{\sigma(r_2)} \chi_A(\delta_{\gg 0}) \Vol(\mathcal{F}_A)\prod_p \Vol(\SO_A(\mathbb{Z}_p)) \frac{\int_{f \in S_2} 2m_2(f,A) df}{\Vol(S_2)}  \left( \prod_{p \neq 2}  (1+r_p(\mathfrak{R})) \right)\\
&= \sum_{\mathcal{G} \in \mathcal{G}_{\mathbb{Z}}} \sum_{A \in \mathcal{G} \cap \mathscr{L}_{\mathbb{Z}}} \frac{2}{\sigma(r_2)}  \chi_A(\delta_{\gg 0}) \Vol(\mathcal{F}_A)\prod_p \Vol(\SO_A(\mathbb{Z}_p)) \frac{\int_{f \in S_2} 2m_2(f,A) df}{\Vol(S_2)}  \left( \prod_{p \neq 2}  (1+r_p(\mathfrak{R})) \right) \\
&= \sum_{\mathcal{G} \in \mathcal{G}_{\mathbb{Z}}} \frac{2}{\sigma(r_2)}  \chi_{\mathcal{G}}(\delta_{\gg 0}) \frac{\int_{f \in S_2} 2m_2(f,\mathcal{G}) df}{\Vol(S_2)} \left( \sum_{A \in \mathcal{G} \cap \mathscr{L}_{\mathbb{Z}}}  \Vol(\mathcal{F}_A)\prod_p \Vol(\SO_A(\mathbb{Z}_p)) \right)  \left( \prod_{p \neq 2}  (1+r_p(\mathfrak{R})) \right)  \\
&= \tau(\SO) \sum_{\mathcal{G} \in \mathcal{G}_{\mathbb{Z}}}  \frac{2}{\sigma(r_2)} \chi_{\mathcal{G}}(\delta_{\gg 0}) \frac{\int_{f \in S_2} 2m_2(f,\mathcal{G}) df}{\Vol(S_2)}  \left( \prod_{p \neq 2}  (1+r_p(\mathfrak{R})) \right) \\
&= \frac{2\tau(\SO)}{2^{r_1+r_2}} \sum_{\mathcal{G} \in \mathcal{G}_{\mathbb{Z}}}  \chi_{\mathcal{G}}(\delta_{\gg 0}) \frac{\int_{f \in S_2} 2m_2(f,\mathcal{G}) df}{\Vol(S_2)}  \left( \prod_{p \neq 2}  (1+r_p(\mathfrak{R})) \right)\\
&= \frac{1}{2^{r_1+r_2-1}} \cdot 2 \left( 2^{n-2} + 2^{\frac{n-2}{2}}\right)  \left( \prod_{p \neq 2}  (1+r_p(\mathfrak{R})) \right) \\
&= \left( 2^{r_2}+\frac{2^{r_2}}{2^{\frac{n-2}{2}}} \right)  \left( \prod_{p \neq 2}  (1+r_p(\mathfrak{R})) \right). \\
\end{align*}

Therefore, we find the following formula for average $2$-torsion in the narrow class group of fields unramified at $2$:
$$\boxed{{\rm Avg}(\Cl_2^+,\mathfrak{R}) = \prod_{p \neq 2} (1+r_p(\mathfrak{R}))\left(1+\frac{2}{2^{\frac{n}{2}}}\right) + \frac{1}{2^{r_2}}} \quad . \\$$

Note that for totally imaginary fields, the narrow class group is the same as the class group and that the formulas do agree in that case!

%\section{On Cohen-Lenstra-Martinet-Malle type heuristics at bad primes} \label{On Cohen-Lenstra-Martinet-Malle type heuristics at bad primes}
%
%To be added...

%We could introduce some heuristics concerning the average number of $p$-torsion elements in the class group of fields whose degree is divisible by $p$.

\bibliographystyle{plain}
\bibliography{bibliography}

\end{document}